\documentclass[11pt]{amsart}
\usepackage{amsmath, url,bm}
\usepackage{amssymb}
\usepackage{enumerate, color, xcolor}
\usepackage{euscript}
\usepackage[all]{xy}
\usepackage{amscd}
\usepackage{xstring}

\theoremstyle{definition}

\newtheorem{lem}{Lemma}
\newtheorem{prop}[lem]{Proposition}
\newtheorem{thm}[lem]{Theorem}
\newtheorem{cor}[lem]{Corollary}

\numberwithin{lem}{section}

\DeclareMathOperator{\Ker}{Ker}
\DeclareMathOperator{\gl}{GL}

\DeclareMathOperator{\pgl}{PGL}
\DeclareMathOperator{\PGL}{PGL}
\DeclareMathOperator{\GL}{GL}

\DeclareMathOperator{\Sym}{Sym}

\newcommand{\C}{\mathbb{C}}
\newcommand{\F}{\mathbb{F}}
\newcommand{\R}{\mathbb{R}}
\newcommand{\Q}{\mathbb{Q}}
\newcommand{\Z}{\mathbb{Z}}
\newcommand{\N}{{\mathbb {N}}}
\newcommand{\A}{{\mathbb {A}}}

\newcommand{\disc}{{\rm Disc}}
\newcommand{\Disc}{{\rm Disc}}

\newcommand{\Sing}{{\rm Sing}}

\newcommand{\fx}{\mathfrak X}

\newcommand{\relmiddle}[1]{\mathrel{}\middle#1\mathrel{}}
\newcommand{\relmid}{\relmiddle{|}}

\newcommand{\Jac}{\mathrm{Jac}}

\newcommand{\Projsp}{\mathbb{P}}

\newcommand{\xa}{X_{1^22}^f}
\newcommand{\xb}{X_{2^2}^f}
\newcommand{\xc}{X_{1^21^2}^f}

\newcommand{\psia}{\psi_{1^22}}
\newcommand{\psib}{\psi_{2^2}}
\newcommand{\psic}{\psi_{1^21^2}}

\newcommand{\yicmtwo}[1]{{\leavevmode\textcolor[rgb]{0.65,0.1,0.85}{#1}}}
\renewcommand{\yicmtwo}[1]{{{#1}}}

\title{Exponential sums over singular binary quartic forms and applications}

\author{Yasuhiro Ishitsuka, Takashi Taniguchi, Frank Thorne, and Stanley Yao Xiao}
\begin{document}

\maketitle

\begin{abstract}
We investigate exponential sums over singular binary quartic forms, proving an explicit formula for the finite field Fourier transform of 
this set. Our formula shares much in common with analogous formulas proved previously for other vector spaces, but also exhibits a striking
new feature: the point counting function $a_p(E) = p + 1 - \#E(\F_p)$ associated to an associated elliptic curve makes a prominent appearance.
The proof techniques are also new, involving techniques
from elementary algebraic geometry and classical invariant theory.
As an application to prime number theory, we demonstrate the existence of
`many' 2-Selmer elements for elliptic curves with
discriminants that are squarefree and have at most four prime factors.
\end{abstract}

\section{Introduction}

Let $V$ be a finite-dimensional $\F_p$-vector space, and let $\Phi_p : V \rightarrow \C$ be any function. In many contexts
it is natural to ask for an equidistribution result on the values of $\Phi_p$, and we may quantify such equidistribution by proving bounds
on its Fourier transform
\[
\widehat{\Phi_p}(v) := p^{-\dim(V)} \sum_{w \in V(\F_p)} \Phi_p(w) e^{2\pi i[w,v] / p}
\qquad
(v\in V^\ast),\]
where $V^\ast$ is the dual space and $[\cdot,\cdot]\colon V\times V^\ast\rightarrow \F_p$ is the canonical pairing.

In many cases one can obtain not only upper bounds, but explicit formulas which illuminate the structure of $V$ and $\Phi_p$.
Here are several examples which have appeared in the literature:

\begin{itemize}
\item
Let $V$ be the space of binary cubic forms. Then, with $G = \textnormal{GL}(2)$, the pair $(G, V)$ is a {\itshape
prehomogeneous vector space}: over an algebraically closed field, the action of $G(k)$ on $V(k)$ has a Zariski open
orbit. There is a natural identification of $V^\ast$ with $V$
for ${\rm char}(k)\neq 3$.

Over $\F_p$, the action of $G(\F_p)$ has six orbits: three singular orbits consisting of $f \in V$ with $\Disc(f) = 0$,
and three nonsingular orbits. In \cite{mori}, Mori proved an explicit formula for $\widehat{\Phi_p}$ for any $G(\F_p)$-invariant
function $\Phi_p$. As a representative example, let $\Phi_p$ be the characteristic function of singular binary forms (i.e., those 
whose discriminants are zero in $\F_p$). Then, for $p \neq 3$, it is immediate from Mori's work that
\begin{equation}\label{eq:cubic_fourier}
\widehat{\Phi_p}(f)=
\begin{cases}
	p^{-1} + p^{-2} - p^{-3} & f=0,\\
	p^{-2} - p^{-3}		& \text{$f \neq 0$ and $\Disc(f) = 0 $},\\
	-p^{-3}		& \text{$\Disc(f) \neq 0$}.
\end{cases}
\end{equation}
We note that: (a) the Fourier transform is only $O(p^{-3})$ on average, since the last case is the generic one, thereby
obtaining better than square root cancellation
of $O(p^{-2.5})$ in $L_1$ norm;
(b) the formula is uniform in $p$; (c) the formula exhibits an
elegant `shape', with the largest Fourier transforms occurring on the most singular orbits.

Similar formulas were obtained in the prehomogeneous case by the second and third authors in \cite{TT_orbital} and by
Ishimoto \cite{ishimoto}.
With $\F_p$ replaced by $\Z/p^2$, additional such formulas 
were obtained in \cite{TT_L} and by Hough \cite{hough}.

\item Using more elaborate algebraic machinery, Fouvry and Katz \cite{FK} 
obtained {\itshape upper bounds} 
for related exponential sums in a much more
general context. As a special case, let $Y$ be a (locally closed) subscheme 
of $\A_{\Z}^n$, and for each prime $p$ let 
$\Phi_p$ be the characteristic function of $Y(\F_p)$.
Fouvry and Katz
produce a filtration of
subschemes $\A_\Z^n \supseteq X_1 \supseteq \cdots \supseteq X_j \supseteq \cdots \supseteq X_n$ of 
increasing codimension, so that successively weaker upper bounds hold on each
$(\A_\Z^n - X_j)(\F_p)$.

These more general results illustrate many of the same features, with an elegant shape, and with the subschemes
$X_i$ being defined over $\Z$ and serving for all primes $p$ simultaneously.

As an interesting application (\cite[Corollary 1.3]{FK}), they proved that a positive proportion of 
primes $p \equiv 1 \pmod 4$ are such that $p + 4$ is squarefree and not the discriminant of a cubic field.

\item 
Again in the prehomogeneous case, Denef and Gyoja \cite{DG} chose
 $\widehat{\Phi_p}$ to be the (non-$G(\F_p)$-invariant) function $\chi(\Disc(v))$, where $\chi$ is a nontrivial Dirichlet character
 $\pmod p$. 
Denef and Gyoja then proved
that the Fourier transform of $\Phi(v)$ is equal to
$\chi^{-1}(\Disc(v))$
times a factor
independent of $v$, somewhat
recalling the shape of Sato's fundamental
theorem \cite[Theorem 4.17]{kimura}\footnote{
The original reference for this theorem is
\cite{sato-ayumi} in Japanese.
We mention that part of this article is translated
into English \cite{sato}, but the fundamental theorem
is not included in there.}
 of prehomogeneous vector spaces over $\R$.

\end{itemize}
An irreducible representation of an algebraic group is called a {\itshape coregular} space if the ring of polynomial invariants is free.
These generalize the prehomogeneous vector spaces,
where the invariant ring has a single generator (the discriminant).
These,  like the prehomogeneous vector spaces, have been the subject
of spectacular parametrization and arithmetic density theorems,
a few of which we will discuss shortly.

In this paper we investigate an exponential sum associated to a coregular space and ask to what extent it enjoys
structure similar to \eqref{eq:cubic_fourier}. In particular, we now let 
$V$ the vector space of binary quartic forms. We say that $f\in V$ is
singular if the discriminant $\Disc(f)$ of $f$ is zero.
We prove the following:

\begin{thm}\label{thm:ff-disc}
Let $p > 3$ be a prime,  and let $\Phi_p\colon V\rightarrow\{0,1\}$
be the characteristic function of singular binary quartic forms.
Then, we have
	\[
		\widehat{\Phi_p}(f)=
		\begin{cases}
			p^{-1}+p^{-2}-p^{-3} & (f = 0) \\
			p^{-2}-p^{-3} & (\mbox{$f$ has splitting type $(1^4)$ or $(1^31)$}) \\
			\chi_{12}(p) (-p^{-3}+p^{-4}) & (\mbox{$f$ has splitting type $(1^21^2)$}) \\
			\chi_{12}(p) (p^{-3}+p^{-4}) & (\mbox{$f$ has splitting type $(2^2)$}) \\
			\chi_{12}(p) p^{-4} & (\mbox{$f$ has splitting type $(1^211)$ or $(1^22)$}) \\
			\left( \frac{-3I(f)}{p}\right) \cdot p^{-4} & (J(f) = 0, I(f) \neq 0) \\
			a(E_f')p^{-4} & (J(f) \neq 0, \Disc(f) \neq 0), \\
		\end{cases}
	\]
with the following notations and conventions:
\begin{itemize} 
\item
As with \eqref{eq:cubic_fourier},
we use a natural identification of $V^\ast$ with $V$;
see \eqref{eq:bilinear} for detail.
The definition of splitting type of $f\in V^\ast=V$
is recalled in Section \ref{sec:background}.
         \item The invariants $I(f)$ and $J(f)$
are naturally associated to a natural action of $\PGL_2$ on $V$
of degree $2$ and $3$,
         as we will recall in \eqref{eq_defI} and \eqref{eq_defJ}.
$\Disc(f)$ is related to these by $\Disc(f)=(4I(f)^3-J(f)^2)/27$.
	\item $E_f'$ in the last line is the elliptic curve over $\F_p$ defined by
	\begin{equation}\label{eq:def_evp}
		y^2 = x^3 - 3I(f)x^2 + J(f)^2.
	\end{equation}
	We as usual define $a(E_f') := p + 1 - \#E'_f(\F_p)$.
	\item $\left( \frac{\cdot}{p} \right)$ is the Legendre symbol,
and $\chi_{12}$ is the primitive Dirichlet character modulo 12,
namely $\chi_{12}(p)$ is $1$ or $-1$ according as $p\equiv\pm1\pmod{12}$ or $p\equiv\pm5\pmod{12}$.
\end{itemize}

\end{thm}
Note that the discriminant of $E_f'$ in the Weierstrass form
\eqref{eq:def_evp} is equal to
$2^4 3^6 J(f)^2\Disc(f)$, which is nonzero.
Again, we obtain an elegant shape,
with the largest values on the most singular elements.
(Note that the Fouvry-Katz bound includes this example,
so that we see the same shape in sharper form.) 
The most novel feature of the formula is the appearance of $a(E_f')$
in the last line, counting points on an elliptic curve.
For all of the prehomogeneous vector spaces we studied,
including \cite{TT_orbital, ishimoto}, the exponential
sums are always polynomials in $p^{-1}$. It is surprising to us
that even though the exponential sum is no longer such a polynomial,
it still admits a `closed' formula.
It is natural  to ask if related phenomena occur with
other coregular spaces, and if this can be predicted by 
the sheaf cohomology machinery of Fouvry-Katz \cite{FK}
and Katz-Laumon \cite{KL}.

We briefly explain our proof of Theorem \ref{thm:ff-disc}. Our proof shares some spirit with that of \cite{TT_orbital} in background, but the method developed in \cite{TT_orbital} is designed for prehomogeneous vector spaces and is not sufficient to prove our theorem in practice. Our main innovation to prove Theorem \ref{thm:ff-disc} is to study certain ``geometric decompositions'' of the singular set. This reduces the proof to counting the numbers of rational points on three projective schemes for each $f$. By appealing to classical invariant theory we determine the cardinalities. One of the three schemes is a genus one curve when $J(f)\Disc(f)\neq0$. We apply a formula from the work of Bhargava and Ho \cite{BH} to find that its Jacobian variety is $E_f'$.
Note that in general this $E_f'$ is neither isomorphic nor isogenous
to the elliptic curve $E_f$ associated with non-singular $f\in V$,
which is the Jacobian variety of the genus one curve $z^2=f(x,y)$.

As a consequence of Theorem \ref{thm:ff-disc},
we obtain outstanding cancellation in $L_1$ norm,
being $O(p^{-7/2})$ on average,
thanks to the Hasse-Weil bound  $|a(E_f')|\leq 2\sqrt p$.
This implies  unusually strong equidistribution for singular
quartic forms in the space of all binary quartic forms. We prove
this equidistribution in an explicit quantitative form
in Corollary \ref{cor:box-estimate}
and Theorem \ref{thm:box-estimate} as
{\itshape level of distribution} results.
Thanks to the Hasse-Weil bound,
Theorem \ref{thm:ff-disc} allows us to prove that
the level of distribution is $>1/3-\epsilon$ for any $\epsilon>0$.

\medskip
{\bf Sieve Application: Lower density of almost prime discriminants.}
We now describe an application to prime number theory.
As mentioned above, Fouvry and Katz proved in \cite{FK} that
a positive proportion of
primes $p \equiv 1 \pmod 4$ are such that $p + 4$ is squarefree and not the discriminant of a cubic field. To do this, they proved
an {\itshape upper bound} on the number of $p \equiv 1 \pmod 4$ such that $p + 4$ is squarefree and {\itshape is} the discriminant of a cubic field,
smaller than known asymptotics for the number of such primes independent of the cubic field condition.

This was accomplished by
sieve methods. If one can obtain asymptotics on the number of integers $n$ divisible by each integer $d$, counted with multiplicity as $n + 4$ ranges over cubic field
discriminants, and obtain 
a good enough cumulative bound on the error term when summed over $d \leq X^{\alpha}$ for large enough $\alpha > 0$, 
then such a result
follows by standard methods. This $\alpha$ is known as a {\itshape level of distribution} for the sieve problem.

To obtain such asymptotics, they relied on the {\itshape Delone-Faddeev correspondence}, in fact due originally to Levi \cite{levi, DF, GGS}, which establishes
a bijection between cubic rings and $\textnormal{GL}_2(\Z)$-equivalence classes of binary cubic forms. The problem
was thus reduced to counting lattice points satisfying various congruence conditions, and obtaining an acceptable bound on the cumulative error terms. 
To improve the error terms, they majorized their overall counting problem from above in a way that simplified it.
Finally, their exponential sum bounds implied equidistribution, which led to good error terms.

\smallskip
Related ideas were further developed in work of the second and third authors \cite{TT_leveldist}, which obtained lower bounds on the number
of cubic and quartic fields having squarefree and almost prime discriminants. Such a result was previously proved for cubic fields
by Belabas and Fouvry \cite{BF}, with a weaker definition of ``almost prime'', and the goal of \cite{TT_leveldist} was to optimize the method
from a quantitative point of view, applying the strong exponential sum bounds implied by the explicit formulas in \cite{TT_orbital}.

\smallskip
In the case of binary quartic forms, work of Birch and Swinnerton-Dyer \cite{BSD} (further developed by Cremona \cite{cremona}),
establishes a bijection between $\textnormal{PGL}_2(\Q)$-orbits of
locally soluble  binary quartic forms and $2$-Selmer
groups of elliptic curves. (See Section \ref{sec:ap} for a precise statement.) This work was then exploited in
spectacular fashion by Bhargava and Shankar \cite{BS2}, who used geometry-of-numbers techniques
to prove that when all elliptic curves are ordered by height, the average size of the $2$-Selmer group ${\rm Sel}_2(E)$ is $3$, and that
the average rank is therefore bounded by $1.5$ (and, in particular, bounded at all).

These results invite applications to $2$-Selmer groups, where we prove:

\begin{thm}\label{thm:ap-2selmer}
We have
\begin{equation}
\sum_{\substack{
	E:\,\text{elliptic curve $/\Q$}\\
	H(E)<X\\
	\Omega(\disc(E))\leq 4\\
	\disc(E)\yicmtwo{:}\, \text{squarefree}
	}}
\left(|{\rm Sel}_2(E)|-1\right)
\gg\frac{X^{5/6}}{\log X}.
\end{equation}
\end{thm}

Here, following \cite{BS2}, we define the {\itshape height} $H(E)$ of an elliptic curve $E$ to be $\max(4|A|^3, 27B^2)$, where we choose the 
unique integral Weierstrass model $y^2 = x^3 + Ax + B$ for $E$
with the property that no prime $p$ satisfies $p^4 \mid A$ and $p^6 \mid B$,
and $\Omega(\disc(E))$ stands for the number of prime factors of $\disc(E)$.
Theorem \ref{thm:ap-2selmer} should be compared with the formula $\sum_{H(E) < X} \left(|{\rm Sel}_2(E)|-1\right) \sim c X^{5/6}$ (for a constant $c>0$), where $E$ runs through all elliptic curves over $\Q$ with $H(E)<X$,
established by Bhargava and Shankar in \cite{BS2}.

Not much is known on the number of elliptic curves with
prime or almost prime discriminants.
For example, it is unknown whether there are infinitely many
elliptic curves with prime discriminants.
To the authors' knowledge, Theorem \ref{thm:ap-2selmer} is the
first case in which it is shown that the sum in (3)
(without the restriction $H(E)<X$) is infinite.

In the language of binary quartic forms, we also show:

\begin{thm}\label{thm:ap-bqf}
The number of $\GL_2(\Z)$-equivalence
classes of irreducible
integral binary quartic forms with height bounded above by $X$,
whose discriminant is squarefree and has at most $4$ prime factors,
is $\gg X^{5/6}/\log X$.
\end{thm}

The height of an integral binary quartic form $f$ 
is defined by $H(f) = \max(|I(f)^3|, J(f)^2/4)$
(again with $I(f)$ and $J(f)$ given in \eqref{eq_defI} and \eqref{eq_defJ},
respectively).
Again, this is the first case that infinitely many
$\GL_2(\Z)$-equivalence classes of irreducible integral binary quartic forms
whose discriminant has at most $4$ prime factors are shown to exist.
In fact, by taking into account the interpretation of orbits of binary quartic forms \cite{BSD}, Theorem \ref{thm:ap-bqf} can be derived from Theorem \ref{thm:ap-2selmer}.

Given Theorem \ref{thm:ff-disc}, the proofs of
Theorems \ref{thm:ap-2selmer} and \ref{thm:ap-bqf}
 follow from methods similar to those in \cite{TT_leveldist},
together with a tail estimate due
to Shankar, Shankar, and Wang \cite{SSW} which is critical to establishing the squarefree condition above.
Because of our sieve, some potential complications arising from 
the $\textnormal{PGL}_2(\Q)$-equivalence (as opposed to $\GL_2(\Z)$-equivalence) can be avoided.

The parameterization of Birch--Swinnerton-Dyer and Cremona is only one of many such associated
to coregular spaces. Of the most immediate interest here are those parametrizations applied by Bhargava and Shankar
\cite{BS3, BS4, BS5} to compute the average sizes of the $3$-, $4$-, and $5$-Selmer groups of elliptic curves and improve
the above stated rank bounds.
We also refer to work of Bhargava and Ho \cite{BH} for a host of further parametrizations associated to genus $1$ curves.

Using the Fouvry-Katz bounds, one may obtain additional applications
along the lines of Theorem \ref{thm:ap-2selmer}, and if results like Theorem \ref{thm:ff-disc} can be obtained for any of these coregular spaces, 
then one may expect concomitant quantitative improvements of these applications.

\medskip{\bf Summary of the paper.} 
In Section \ref{sec:background}
we recall some basic background on the space of
binary quartic forms $V$ and its $\PGL_2$ action.
Section \ref{sec:expsum}
is devoted to proving Theorem \ref{thm:ff-disc}.
We first reformulate
the problem in terms of counting points
on a variety $X^f\subset \Projsp(V)$ which we define in \eqref{eq:def_xf}.
We next introduce a ``geometric decomposition'' of $X^f$;
we construct three morphisms from lower dimensional projective spaces
into $\Projsp(V)$, and express $\# X^f(\F_p)$
in terms of point counts of their inverse images,
or equivalently in terms of subschemes
of the respective domain spaces.
We then demonstrate the counts for the three schemes,
and complete the proof of Theorem \ref{thm:ff-disc}.

We then prove Theorems \ref{thm:ap-2selmer} and \ref{thm:ap-bqf},
following the approach of \cite{TT_leveldist}.
We begin in Section \ref{sec:box_estimate} with a `box estimate',
bounding the total of $|\widehat{\Phi_q}(f)|$, as $f$ ranges over a box
except for the origin
and $q$ ranges over squarefree integers in a dyadic interval.
This immediately yields a `level of distribution' estimate
for the function $\Phi_q$.
In Section \ref{sec:ap} we complete the proofs of Theorems
\ref{thm:ap-2selmer} and \ref{thm:ap-bqf}.
We first recall 
the parametrization of $2$-Selmer elements of
elliptic curves in terms of $V$, closely following
Bhargava and Shankar \cite{BS2}.
We resolve a few technical issues
on the reductions at primes $2$ and $3$,
and bound the multiplicity of the integral orbits inside the rational
orbits in our count, completing the proof of Theorem \ref{thm:ap-2selmer}.
Theorem \ref{thm:ap-bqf} is proved within this process.

Throughout, for real-valued functions $f$ and $g$ whose domain includes all sufficiently large real numbers,
we write $f \gg g$ or $g \ll f$ if there exist positive constants $c,d$
such that $f(X)>d|g(X)|$ for all $X>c$.

\section{The space of binary quartic forms}\label{sec:background}

In this section we recall basic facts and notation about binary quartic forms. (See Section \ref{sec:ap} for background on the 
parametrization of $2$-Selmer groups of elliptic curves.)

Let $V$ be the space of binary quartic forms; that is, for any ring $R$,
\begin{align*}
V(R)	&= \left\{ f(x,y) = a_0x^4 + a_1x^3y + a_2x^2y^2 + a_3xy^3 + a_4y^4 \relmid a_0, a_1, a_2, a_3, a_4 \in R \right\}.
\end{align*}
We consider the natural action of $\gl_2$ on $V$:
\begin{equation}\label{eq:DefOfAction}
	\begin{pmatrix} a & b \\ c & d \end{pmatrix} \cdot f(x,y) = f(ax + cy, bx + dy).
\end{equation}
For this action, there are two fundamental invariants:
\begin{align}
	I(f) &= 12a_0a_4 - 3a_1a_3 + a_2^2, \label{eq_defI} \\
	J(f) &= 72a_0a_2a_4 + 9 a_1a_2a_3 - 27(a_0a_3^2 + a_1^2a_4) - 2a_2^3, \label{eq_defJ}
\end{align}
where $f(x,y) = a_0x^4 + a_1x^3y + a_2x^2y^2 + a_3xy^3 + a_4y^4$.
For $\alpha \in R^\times$ and $g \in \gl_2(R)$, we have
\begin{align*}
	I (\alpha g \cdot f) &= \alpha^2 (\det g)^4 f, \\
	J (\alpha g \cdot f) &= \alpha^3 (\det g)^6 f.
\end{align*}
Thus in particular the value $\left( \frac{-3 I(v)}{p} \right)$ in Theorem \ref{thm:ff-disc} depends only on the $\F_p^{\times}\times\gl_2(\F_p)$-orbit
of $v$.
Another distinguished invariant, the \textit{discriminant} $\Disc(f)$ of a binary quartic form $f(x,y)$ is defined as
\begin{equation}\label{eq:def_disc}
	\Disc(f) = \frac{4I(f)^3 - J(f)^2}{27}.
\end{equation}
Note that $\Disc(\alpha g \cdot f) = \alpha^6 (\det g)^{12} \Disc(f)$ for $\alpha \in k^\times$ and $g \in \gl_2(k)$.
The vanishing of the discriminant is equivalent that $f$ has a multiple factor.
We also consider another action of $\pgl_2$, induced by the twisted action
\begin{equation}
	\begin{pmatrix} a & b \\ c & d \end{pmatrix} \circ f(x,y) =
	\frac{1}{(ad - bc)^2} f(ax + cy, bx + dy).
\end{equation}
Under this action of $\pgl_2$, the functions $I, J$ and $\Disc$ are actually invariants.

\yicmtwo{When $2$ and $3$ are not zerodivisors,} we define a bilinear form on $V$ 
\yicmtwo{valued in $12^{-1}R$} by
\begin{equation}\label{eq:bilinear}
	\left[ f, h \right] = a_0b_0 + \frac{a_1b_1}{4} + \frac{a_2b_2}{6} + \frac{a_3b_3}{4} + a_4b_4,
\end{equation}
where $f = f(x,y) = a_0x^4 + a_1x^3y + a_2x^2y^2 + a_3xy^3 + a_4y^4$
and $h = h(x,y) = b_0x^4 + b_1x^3y + b_2x^2y^2 + b_3xy^3 + b_4y^4$.
This bilinear form satisfies
\(
	[g \cdot f, h] = [f, g^T \cdot h].
\)
When $R$ is a finite field of characteristic 
$p > 3$ this form induces an isomorphism $V^* \to V$, and when $R = \Z$ it induces an injection 
$V^*(\Z) \hookrightarrow V(\Z)$.

When $R$ is a field, we recall from \cite{BS2} that $0 \neq f \in V(R)$
has splitting type $(d_1^{e_1} \cdots d_r^{e_r})$ if $f$ has $r$ distinct irreducible factors, with the $i$th factor
being of degree $d_i$ and appearing to multiplicity $e_i$. The possible splitting types are: $(1111)$, $(211)$,
$(31)$, $(22)$, $(4)$, $(1^2 11)$, $(1^2 1^2)$, $(2^2)$, $(1^3 1)$, and $(1^4)$. The first five splitting types,
with $\Disc(f) \neq 0$, are called {\itshape nondegenerate} (or nonsingular), and the remaining five (and the zero form) are called
{\itshape degenerate} (or singular).

For finer classification results concerning $\GL_2$-orbits of binary cubic forms when $R$ is a finite field -- of particular
interest in the nonsingular case -- we refer to Kamenetsky \cite{kam} and Kaipa, Patanker, and Pradhan \cite{KPP}.
The thesis \cite{kam} also includes some preliminary work which might also prove useful in the computation of the Fourier transforms
of other $\GL_2(\F_p)$-invariant functions.

\section{Exponential sums and counting points on projective schemes}\label{sec:expsum}
In this section we prove Theorem \ref{thm:ff-disc}.
In Section \ref{subsec:contraction}
we reduce the problem to counting the $\F_p$-rational points
of a projective variety $X^f$, defined in \eqref{eq:def_xf}.
In Section \ref{subsec:decomposition} we introduce
a ``geometric decomposition'' of $X^f$, which plays
a crucial role in our proof.
This reduces the problem to counting rational points of
three projective schemes $\xa$, $\xb$ and $\xc$.
In the next three subsections we demonstrate the counts
for these schemes. Finally in Section \ref{subsec:proof-expsum}
we complete the proof of Theorem \ref{thm:ff-disc}.
We assume that $p \neq 2, 3$ throughout this section.

\subsection{Scalar contraction}\label{subsec:contraction}
To begin the proof of Theorem \ref{thm:ff-disc}, 
we first reduce the proof to counting rational points on projective \yicmtwo{schemes}.
We regard $\Projsp(V)$ as the set of 1-dimensional subspaces of $V$, 
and denote the one-dimensional space spanned by a non-zero vector $h \in V$ as $\overline{h} \in \Projsp(V)$.
(We will later do the same for other vector spaces as well.)
Writing $\bm{1}_{\Disc}$ for the characteristic function of those $h \in V$ with $\Disc(h) = 0$, 
we thus have
\begin{equation}\label{eq:expsum1}
	p^5 \widehat{\Phi_p}(f) = \sum_{h \in V} \bm{1}_{\Disc}(h)\psi([h,f]) = 
	1 + \sum_{\overline{h} \in \Projsp(V)} \sum_{h \in \overline{h}} \yicmtwo{\bm{1}_{\Disc}(h)} \psi([h,f]).
\end{equation}
The vanishing of $\Disc(h)$ and $[h, f]$ do not depend on the choice of $h \in \overline{h}$.
This allows us to define
\begin{align}
	X &:= \left\{ \overline{h} \in \Projsp(V) \mid \Disc(h) = 0 \; (h \in \overline{h}) \right\}, \\
	X^f &:= \left\{ \overline{h} \in \Projsp(V) \mid \Disc(h) = [h, f] = 0 \; (h \in \overline{h}) \right\}. \label{eq:def_xf}
\end{align}
We abuse the same notation $\bm{1}_{\Disc}$ to write the characteristic function of
the algebraic subset $X$ in $\Projsp(V)$.
Using them and the equality
\[
	\sum_{\substack{h \in \overline{h} \\ h \neq 0}} \psi([h,f]) = 
	\begin{cases}
		p-1 & ([h,f] = 0) \\
		-1 & ([h,f] \neq 0),
	\end{cases}
\]
we have
\begin{align}
p^5 \widehat{\Phi_p}(f) 
	&=  1 + (p-1) \sum_{\overline{h} \in \Projsp(V), [\overline{h},f] = 0} \bm{1}_{\Disc}(\overline{h}) 
	- \sum_{\overline{h} \in \Projsp(V), [\overline{h},f] \neq 0} \bm{1}_{\Disc}(\overline{h}) \nonumber \\
	&= 1 + p \#X^f(\F_p) - \# X(\F_p). \label{eq:expo_breakdown}
\end{align}
By the relation $[g \cdot a, b] = [a, g^T \cdot b]$, we see that $\widehat{\Phi_p}(f)$ 
depends only on the $\PGL_2(\F_p)$-orbit of $f$, and we also note that $\widehat{\Phi_p}(\lambda f) = \widehat{\Phi_p}(f)$
for all $\lambda \in \F_p^{\times}$. (Equivalently, we could expand the $\PGL_2$ action to a natural $\GL_1 \times \PGL_2$ or  $\GL_1 \times \GL_2$ action,
with respect to which $\Phi_p$ and thus $\widehat{\Phi_p}$ would still be invariant.)

To calculate $\# X(\F_p)$, we apply the following:
\begin{lem}
Let $p$ be a prime and let $n \geq 3$. Then, 
the number of squarefree $n$-ic binary forms over $\F_p$ is $p^{n + 1} (1 - p^{-1})(1 - p^{-2})$.
\end{lem}
\begin{proof} Dehomogenizing and dividing by the leading constant, we obtain a $(p - 1)$-to-$1$ surjection onto the set of monic squarefree polynomials of degree $n$ or $n - 1$. 
It is classically known that, for $k \geq 2$, there are $p^k (1 - p^{-1})$ monic squarefree polynomials of degree $k$ (see e.g. \cite{yuan} for a proof), and
\[
(p - 1) \left[ p^n (1 - p^{-1}) + p^{n - 1} (1 - p^{-1}) \right] = p^{n + 1} (1 - p^{-1})(1 - p^{-2}).
\]
\end{proof}
Therefore, there are $p^4 + p^3 - p^2$ singular quartic forms.
This proves Theorem \ref{thm:ff-disc} for $f=0$, and we also have
\begin{equation}
		\# X(\F_p) = \sum_{\overline{h} \in \Projsp(V)} \bm{1}_{\Disc}(\overline{h}) = \frac{1}{p - 1} (p^4 + p^3 - p^2 - 1) = p^3 + 2p^2 + p + 1.
\end{equation}
Thus to prove Theorem \ref{thm:ff-disc},
it is enough to calculate $\# X^f(\F_p)$ for $0\neq f \in V(\F_p)$.

\subsection{Geometric decomposition of the singular set}\label{subsec:decomposition}
Let $\Sym^k\F_p^2$ denote the space of binary $k$-ic forms in variables $x$ and $y$ over $\F_p$. We consider its coordinates with the basis $x^k, x^{k-1}y,\dots, y^k$ to identify $\Sym^k\F_p^2$ with $\F_p^{k+1}$.

Now we consider the problem: for each degenerate point $f \in \Projsp(V)=\Projsp(\Sym^4\F_p^2)$ (which we will typically lift to an element of $V$),
can we represent the form $f$ as $l^2q$ for a linear form $l=l(x,y)$ and 
a quadratic form $q = q(x,y)$?  If we can, in how many ways can we represent it?

More formally, we consider the following morphism:
\begin{align*}
	\psia \colon \Projsp(\Sym^1\F_p^2) \times \Projsp(\Sym^2 \F_p^2) &\to \Projsp(\Sym^4 \F_p^2) = \Projsp(V) \\
	(s_0x + s_1y, t_0x^2 + t_1xy + t_2y^2) &\mapsto (s_0x + s_1y)^2(t_0x^2 + t_1xy + t_2y^2),
\end{align*}
In terms of coordinates, the map is described as
\[
	\psia([s_0:s_1], [t_0:t_1:t_2]) = [s_0^2t_0: s_0^2t_1 + 2s_0s_1t_0 : s_0^2t_2+ 2s_0s_1t_1 + s_1^2t_0 : 2s_0s_1t_2 + s_1^2t_1 : s_1^2t_2].
\]
The image is contained in the degenerate locus, and we can calculate the cardinalities of the inverse images:
\begin{itemize}
	\item The case $(1^4)$: for a form $h = l_0^4 \in \Projsp(\Sym^4 \F_p^2)$, 
	the only way is $l = l_0 \in \Projsp(\F_p^2), q = l_0^2 \in \Projsp (\Sym^2 \F_p^2)$.
	\item The case $(1^31)$: for a form $h = l_0^3l_1$, the only way is $l = l_0, q = l_0l_1$.
	\item The case $(1^21^2)$: for a form $h = l_0^2l_1^2$,
	there are two ways: $l = l_0, q = l_1^2$, and $l = l_1, q = l_0^2$.
	\item The case $(2^2)$: zero, since the forms are not divisible by a linear form over $\F_p$.
	\item The case $(1^211)$: for a form $h = l_0^2l_1l_2$, the only way is $l = l_0, q = l_1l_2$.
	\item The case $(1^22)$: for a form $h = l_0^2q_0$, the only way is $l = l_0, q = q_0$.
\end{itemize}

Thus, through the morphism $\psia$, we count the case $(1^21^2)$ doubly,
and do not count the case $(2^2)$.
To flatten this multiplicity, we consider two other morphisms:
\begin{align*}
	\psib \colon \Projsp(\Sym^2 \F_p^2) &\to \Projsp(\Sym^4 \F_p^2) = \Projsp(V) \\
	t_0x^2 + t_1xy + t_2y^2 &\mapsto (t_0x^2 + t_1xy + t_2y^2)^2
\end{align*}
and 
\begin{align*}
	\psic \colon \Projsp(\Sym^1\F_p^2) \times \Projsp(\Sym^1\F_p^2) &\to \Projsp(\Sym^4 \F_p^2) = \Projsp(V) \\
	(s_0x + s_1y, t_0x + t_1y) &\mapsto (s_0x + s_1y)^2(t_0x + t_1y)^2.
\end{align*}
We can count $\#\psib^{-1}(h)$ and $\#\psic^{-1}(h)$ similarly.
The following table summarizes the results.
Note that $\#\psi_i^{-1}(h)$
depends only on the splitting type of the binary quartic $h$ for each $1 \le i \le 3$.

\begin{table}[ht]
\centering\begin{tabular}{c||ccc|c}
	\hline
	splitting type of $h$  & $\#\psia^{-1}(h)$ & $\#\psib^{-1}(h)$ & $\#\psic^{-1}(h)$ & $\bm{1}_{\Disc}(h)$\\
	\hline
	\hline
	non-degenerate 	& 0 & 0 & 0 & 0 \\
	\hline
	$(1^4)$ 	& 1 & 1 & 1 & 1 \\
	\hline
	$(1^31)$ 	& 1 & 0 & 0 & 1 \\
	\hline
	$(1^21^2)$ 	& 2 & 1 & 2 & 1 \\
	\hline
	$(2^2)$ 	& 0 & 1 & 0 & 1 \\
	\hline
	$(1^211)$ 	& 1 & 0 & 0 & 1 \\
	\hline
	$(1^22)$ 	& 1 & 0 & 0 & 1 \\
	\hline
\end{tabular} 
\caption{Values of $\#\psi_\sigma^{-1}(h)$ and $\bm{1}_{\Disc}(h)$}
\end{table}
From this table we have the following identity:
\begin{equation}\label{eq:1_breakdown}
	\bm{1}_{\Disc}(\overline{h}) = \#\psia^{-1}(\overline{h}) + \#\psib^{-1}(\overline{h}) - \#\psic^{-1}(\overline{h}).
\end{equation}
Thus it is enough to count the sums
\[
	\sum_{\overline{h} \in \Projsp(V), [\overline{h},f]=0} \#\psi_\sigma^{-1}(\overline{h})
\]
for $\sigma\in\{1^22, 2^2, 1^21^2\}$.
They are equal to the number of the $\F_p$-rational points on
\begin{align}
\xa	&:= \left\{(l,q) \in \Projsp(\Sym^1\F_p^2) \times \Projsp(\Sym^2 \F_p^2)
		\relmid [l^2q, f] = 0 \right\}, \label{eq:xa}\\
\xb	&:= \left\{q \in \Projsp(\Sym^2 \F_p^2)
		\relmid [q^2, f] = 0 \right\}, \label{eq:xb}\\
\xc	&:= \left\{(l_1,l_2) \in \Projsp(\Sym^1\F_p^2) \times \Projsp(\Sym^1\F_p^2)
		\relmid [l_1^2l_2^2, f] = 0\right\},\label{eq:xc}
\end{align}
respectively. We refer to them as the $1^22$-scheme, $2^2$-scheme, and $1^21^2$-scheme, respectively. They are varieties for most $f$, but in some degenerate cases, they can be reducible or non-reduced as we will see in the concrete analysis.

We introduce the action of $\GL_2$ on 
the space of binary linear forms $\Sym^1(\F_p^2)$ and binary quadratic forms $\Sym^2(\F_p^2$) 
in the same way as \eqref{eq:DefOfAction}.
Then the morphisms $\psi_\sigma$ are equivariant.
Also since $[gf, h] = [f, g^{T}h]$,
if $g\in\gl_2(\F_p)$, then
$X_\sigma^f$ and $X_\sigma^{gf}$ are isomorphic
as varieties over $\F_p$. In particular the sets of
their $\F_p$-valued points have the same cardinalities.

From now, we write
\begin{equation}
f = f(x,y) = a_0x^4 + a_1x^3y + a_2x^2y^2 + a_3xy^3 + a_4y^4.
\end{equation}

\subsection{Counting points on the $1^22$-scheme}
In this subsection we determine $\#\xa(\F_p)$.
First we investigate the degeneracy of $\xa$.
The required equality is represented in terms of coordinates as
\[
	\left( a_0s_0^2 + \frac{1}{2}a_1s_0s_1 + \frac{1}{6}a_2s_1^2 \right)t_0 
	+ \left( \frac{1}{4}a_1s_0^2 + \frac{1}{3}a_2s_0s_1 + \frac{1}{4}a_3s_1^2 \right) t_1 + 
	\left( \frac{1}{6}a_2s_0^2 + \frac{1}{2}a_3s_0s_1 + a_4s_1^2 \right) t_2 = 0.
\]
We choose $s = [s_0 : s_1] \in \Projsp(\F_p^2)$.
If one of the coefficients of $t_0, t_1$, and $t_2$ does not vanish,
then the equation defines a linear subspace in $\Projsp(\Sym^2 \F_p^2)$,
and there are $\#\Projsp^1(\F_p) = p+1$ choices for $[t_0:t_1:t_2] \in \Projsp(\Sym^2 \F_p^2)$.
Otherwise, there are no restrictions to choose $[t_0: t_1 : t_2]$,
and $\#\Projsp^2(\F_p) = p^2 + p+1$ choices for $[t_0:t_1:t_2]$.
Thus it is enough to know when all the three coefficients vanish.

We find that the coefficients are the
second partial derivatives of $f(x,y)$ multiplied by $1/12$:
\begin{align*}
	\frac{1}{12} \left. \frac{\partial^2}{\partial x^2}f(x,y) \right|_{x=s_0, y=s_1} &= a_0s_0^2 + \frac{1}{2}a_1s_0s_1 + \frac{1}{6}a_2s_1^2, \\
	\frac{1}{12} \left. \frac{\partial^2}{\partial x \partial y}f(x,y) \right|_{x=s_0, y=s_1} &= \frac{1}{4}a_1s_0^2 + \frac{1}{3}a_2s_0s_1 + \frac{1}{4}a_3s_1^2, \\
	\frac{1}{12} \left. \frac{\partial^2}{\partial y^2}f(x,y) \right|_{x=s_0, y=s_1} &= \frac{1}{6}a_2s_0^2 + \frac{1}{2}a_3s_0s_1 + a_4s_1^2.
\end{align*}
With Euler's formula
\[
	x\frac{\partial}{\partial x}f + y\frac{\partial}{\partial y}f = \deg(f) f,
\]
if all three coefficients vanish, then we have
\[
	f(s_0, s_1) = \frac{1}{4}\frac{\partial f}{\partial x}(s_0, s_1) = \frac{1}{12}\frac{\partial^2 f}{\partial x^2}(s_0, s_1) = 0.
\]
Hence, if $s_1 \neq 0$, then the polynomial $f(x,1)$ has a triple root at $x = s_0/s_1$.
Similarly, if $s_0 \neq 0$, then the polynomial $f(1,y)$ has a triple root at $y = s_1/s_0$.
Thus, if all the three coefficients vanish at $s = [s_0:s_1]$, then 
$f(x,y)$ has a factor $-s_1x + s_0y$ of at least multiplicity three.
Conversely, if $f(x,y)$ has a triple or quadruple root,
we can see that the second partial derivatives vanish at the root.
Hence the vanishing occurs just when $f$ is of the type $(1^31)$ or $(1^4)$,
and the point $s = [s_0: s_1]$ corresponds to the triple or quadruple root.

Now we determine the number of $\F_p$-rational points on $\xa$.
If $f$ is neither of the type $(1^31)$ nor $(1^4)$,
then for any $s = [s_0 : s_1]$ there are $p+1$ choices of $[t_0:t_1:t_2]$.
Hence in this case we have \[\#\xa(\F_p) = (p+1)^2.\]

If $f$ is of the type $(1^31)$ or $(1^4)$,
then for the point $[s_0: s_1] \in \Projsp^1(\F_p)$ corresponding to the multiple root of $v$, there are $p^2+p+1$ choices of $[t_0:t_1:t_2]$.
For other $[s_0:s_1]$, there are $p+1$ choices of $[t_0:t_1:t_2]$.
Hence \[\#\xa(\F_p) = 1 \cdot (p^2+p+1) + p \cdot (p+1) = 2p^2 + 2p + 1.\]

To summarize, we have:

\begin{prop}
	For $f \in \Projsp(V)$, we have
	\[
		\#\xa(\F_p) = \begin{cases}
			(p+1)^2 & \mbox{splitting type of $f$ is neither $(1^31)$ nor $(1^4)$,} \\
			2p^2 + 2p + 1 & \mbox{splitting type of $f$ is either $(1^31)$ or $(1^4)$}.
		\end{cases}
	\]
\end{prop}

\subsection{Counting points on the $2^2$-scheme}
In this subsection we determine $\#\xb(\F_p)$.
The \yicmtwo{subscheme} $\xb \subset \Projsp(\Sym^2\F_p^2)$
is defined by a quadratic form.
To determine the number of rational points, we have to consider two conditions:
the degeneracy of the quadratic form (equivalently, the dimension of the singular locus) and the
splitting of irreducible components if $\xb$ is not geometrically irreducible.

First we study the degeneracy of $\xb$.
The defining equation of $\xb$ is represented in terms of coordinates as
\begin{equation}\label{case2surf}
	a_0t_0^2 + \frac{1}{2}a_1t_0t_1 + \frac{1}{6}a_2(2t_0t_2 + t_1^2) + \frac{1}{2}a_3t_1t_2 + a_4t_2^2 = 0,
\end{equation}
or equivalently, 
\[
	\begin{pmatrix}
		t_0 & t_1 & t_2
	\end{pmatrix}
	\begin{pmatrix}
		a_0 & \frac{1}{4}a_1 & \frac{1}{6}a_2 \\
		\frac{1}{4}a_1 & \frac{1}{6}a_2 & \frac{1}{4}a_3 \\
		\frac{1}{6}a_2 & \frac{1}{4}a_3 & a_4
	\end{pmatrix}
	\begin{pmatrix}
		t_0 \\ t_1 \\ t_2
	\end{pmatrix} = 0.
\]
The Gram matrix of this quadratic form
\[
	M_f = \begin{pmatrix}
		a_0 & \frac{1}{4}a_1 & \frac{1}{6}a_2 \\
		\frac{1}{4}a_1 & \frac{1}{6}a_2 & \frac{1}{4}a_3 \\
		\frac{1}{6}a_2 & \frac{1}{4}a_3 & a_4
	\end{pmatrix}
\] 
 is called the \textit{catalecticant matrix} of $f = f(x,y)$, and its determinant is computed as
\begin{align*}
	\frac{1}{6}a_0a_2a_4 + \frac{1}{48}a_1a_2a_3 - \frac{1}{16}(a_0a_3^2 + a_1^2a_4) - \frac{1}{216}a_2^3 
	= \frac{1}{432}J(f).
\end{align*}
Hence the plane quadric $\xb \subset \Projsp(\Sym^2\F_p^2)$ defined by \eqref{case2surf} is nonsingular if and only if
$J(f) \neq 0$.

To investigate the degenerate case, we consider the adjugate matrix of the catalecticant matrix $M_f$ of $f$.
It is described as
\begin{equation}\label{AdjCat}
	\widetilde{M_f} = \frac{1}{144}\begin{pmatrix}
		24a_2a_4 - 9a_3^2 & -36a_1a_4 + 6a_2a_3 & 9a_1a_3 - 4a_2^2 \\
		-36a_1a_4 + 6a_2a_3 & 144a_0a_4-4a_2^2 & -36a_0a_3 + 6a_1a_2 \\
		9a_1a_3 - 4a_2^2 & -36a_0a_3 + 6a_1a_2 & 24a_0a_2 - 9a_1^2
	\end{pmatrix}.
\end{equation}
The singular locus of $\xb$ is equal to $\Ker(M_f)$, and hence
is of dimension greater than $0$ if and only if the corank of the catalecticant matrix $M_f$ is greater than 1.
It is also equivalent that $\widetilde{M}_f$  is equal to the zero matrix.
(Note that we are assuming $f\neq0$.)
By some computation, it is equivalent that it is of the type $(1^4)$.
Summing up, we have the following classification:
\begin{enumerate}[{(1)}]
	\item $\xb$ is nonsingular: it is equivalent to $J(f) \neq 0$.
	\item $\Sing(\xb)$ consists of one point: it is equivalent to $J(f) = 0$ but $f$ is not of splitting type $(1^4)$. \label{CaseOfTwoLines}
	\item $\Sing(\xb)$ consists of a line: it is equivalent that $f$ is of splitting type $(1^4)$.
\end{enumerate}

In the case \eqref{CaseOfTwoLines}, $\xb$ consists of two lines over $\overline{\F_p}$, and we have to consider whether they split over $\F_p$ or not.
We need the following lemma:
\begin{lem}\label{lem:WaringDecomp}
	For a binary quartic form $f \in V$, the following are equivalent:
	\begin{itemize}
		\item The invariant $J(f) = 432 \det M_f$ vanishes.
		\item The quartic $f$ is of the type $(1^31)$,
		or there are two linear forms $l_0, l_1$ 
		over the algebraic closure $\overline{\F}_p$ such that
		\(
			f = l_0^4 + l_1^4.
		\)
	\end{itemize}
	Moreover, if $f$ is non-degenerate, $l_0, l_1$ are linearly independent over $\overline{\F}_p$
	and the summands $l_0^4, l_1^4$ are uniquely determined by $f$ up to permutation.
	If $f \neq 0$ is degenerate, then $f$ is of the type $(1^31)$ or $(1^4)$.
\end{lem}

This is a classical result of invariant theory:
for the equivalence part over a field of characteristic zero, 
see \cite[Theorem 5.3]{KR84}.
For the convenience of the reader, we give a proof.

\begin{proof}
	Throughout, we work not over $\F_p$ but over its algebraic closure $\overline{\F}_p$.
	
	First, we assume that $f$ can be represented as $l_0^4 + l_1^4$.
	If $l_0, l_1$ are linearly independent,
	by changing coordinates we may assume that $l_0(x,y) = x$ and $l_1(x,y) = y$.
	Then since
	\[
		M_{x^4 + y^4} = \begin{pmatrix}
		1 & 0 & 0 \\
		0 & 0 & 0 \\
		0 & 0 & 1
		\end{pmatrix},
	\]
	we have $J(x^4 + y^4) = 432 \det M_{x^4 + y^4} = 0$.
	Similarly, if $l_0, l_1$ are linearly dependent, then we may assume that $f = x^4$, and 
	if $f$ is of type $(1^31)$ we may assume that $f = x^3y$.
	The corresponding matrices are 
	\[
		M_{x^4} = \begin{pmatrix}
		1 & 0 & 0 \\
		0 & 0 & 0 \\
		0 & 0 & 0
		\end{pmatrix} \mbox{ and } M_{x^3 y} = \begin{pmatrix}
		0 & 1/4 & 0 \\
		1/4 & 0 & 0 \\
		0 & 0 & 0
		\end{pmatrix},
	\]
	and in each case we have $J(f) = 0$.
	
	Conversely, if $J(f) = 0$, we must have
	\[
		\begin{pmatrix} u_0 & u_1 & u_2 \end{pmatrix} M_f = \bm{0}, \quad
		M_f \begin{pmatrix} u_0 \\ u_1 \\ u_2 \end{pmatrix} = \bm{0}
	\]
	for some nonzero vector $(u_0, u_1, u_2)$, and writing $u = u(x,y) = u_0x^2 + u_1xy + u_2y^2$, this is seen to be equivalent to the condition
	that the map $q \mapsto [f, uq]$ be identically zero as $q$ ranges over binary quadratic forms.
		
	If $u(x,y)$ is the square of a linear form,
	by changing coordinates we may assume that $u(x,y) = x^2$, 
	and we have
	\begin{align*}
		a_0 = \frac{a_1}{4} = \frac{a_2}{6} = 0.
	\end{align*}
	This says that $f(x,y) = a_3xy^3 + a_4y^4$ has a triple factor $y$;
	in other words, $f$ is of the form $(1^31)$ or $(1^4)$.
	These are degenerate.
	
	Otherwise, $u(x,y) = u_0x^2 + u_1xy + u_2y^2$ is a product of two non-parallel linear forms.
	By changing coordinates we may assume that $u(x,y) = xy$, and we obtain that 
	$a_1 = a_2 = a_3 = 0$, and conclude that
	$f(x,y) = a_0x^4 + a_4y^4$.
	Since $\Disc(f) = 256 a_0^3a_4^3$, if $f$ is degenerate then either $a_0$ or $a_4$ is zero, and thus $f$ is of the type $(1^4)$.
	
	The claim about linear independence is clear, so we conclude by proving the uniqueness claim. Assume that
	\[
	f = x^4 + y^4 = (c_0x + c_1y)^4 + (d_0x + d_1y)^4
	\]
	where $c_0x + c_1y, d_0x + d_1y$ are linearly independent over $\overline{\F}_p$.
	From the former expression, we compute that $u(x,y) = xy$ up to a constant.
	From the latter expression, we may compute that $u(x,y) = (c_1x - c_0y) (d_1x - d_0y)$ up to a constant.
	This shows that the pair of points $\{[c_0 : c_1], [d_0: d_1]\}$ coincides with $\{[1:0], [0:1]\}$ 
	up to permutation.
	We may assume that $c_1 = d_0 = 0$, then we find $c_0^4 = d_1^4 = 1$.
	This shows the desired uniqueness.
\end{proof}

The above lemma considers $f$ over the algebraically closed field $\overline{\F}_p$.
What if we consider it over $\F_p$?
First we treat the case $J(f) = 0 $ and $\Disc(f) \neq 0$,
which we call the \textit{semi-degenerate case}.
By the uniqueness of the pair $\{l_0^4, l_1^4\}$,
there are two possibilities considering over $\F_p$:

\begin{enumerate}[{(i)}]
	\item The polynomials $l_0^4$, $l_1^4$ are defined over $\F_p$.
	\item The polynomials $l_0^4$, $l_1^4$ are defined over $\F_{p^2}$ and are Galois conjugate.
\end{enumerate}

The former case means that there are linear forms $l_0', l_1'$ over $\F_p$ and nonzero constants $a, b \in \F_p$ so that
\[
	f(x,y) = al_0'(x,y)^4 + bl_1'(x,y)^4.
\]
In the latter case, $l_0^4$ and $l_1^4$ are Galois conjugates, hence we can write
\begin{align*}
	f(x,y) &= (a + \sqrt{\beta} b)(l_0'' + \sqrt{\beta} l_1'')^4 + (a - \sqrt{\beta} b)(l_0'' - \sqrt{\beta} l_1'')^4,
\end{align*}
where $a, b$ are constants in $\F_p$ with $(a,b) \neq (0,0)$, $l_0'', l_1''$ are linear forms over $\F_p$, and $\beta\in\F_p^\times$ is any quadratic nonresidue.
By a linear change of coordinates, we may assume
$l_0' = x, l_1'= y$
and
$l_0'' = x, l_1''= y$.

We first consider the case $f(x,y) = ax^4 + by^4$.
The catalecticant matrix is
\[
	M_f = \begin{pmatrix}
		a & 0 & 0 \\
		0 & 0 & 0 \\
		0 & 0 & b
	\end{pmatrix},	
\]
and the defining equation of $\xb$ is
\[
	\begin{pmatrix}
		t_0 & t_1 & t_2
	\end{pmatrix}
	\begin{pmatrix}
		a & 0 & 0 \\
		0 & 0 & 0 \\
		0 & 0 & b
	\end{pmatrix}
	\begin{pmatrix}
		t_0 \\ t_1 \\ t_2
	\end{pmatrix}
	= at_0^2 + bt_2^2 = 0.
\]
By this equation, we can see that
\begin{itemize}
	\item the singular point of $\xb$ is described as $t_0 = t_2 = 0$, that is, $[0:1:0]$, and
	\item $\xb$ splits if and only if the discriminant (as a quadratic form of $t_0$ and $t_2$)
	$-4ab$ is a square in $\F_p^{\times}$.
\end{itemize}
To restate this condition,
we introduce a covariant $\mathrm{He}_f(x,y)$.
It is defined as
\begin{equation}\label{eq:def_he}
	\mathrm{He}_f(x,y) = -\det \begin{pmatrix}
		f_{xx} & f_{xy} \\
		f_{xy} & f_{yy}
	\end{pmatrix}.
\end{equation}
In the present case, it is computed as 
\begin{equation}\label{eq:he_def}
	\mathrm{He}_f(x,y) = -\det \begin{pmatrix}
		12 ax^2 & 0 \\
		0 & 12 by^2
	\end{pmatrix} = -4ab \cdot (6xy)^2,
\end{equation}
thus the above statement about the splitting of $\xb$ is equivalent to
\begin{itemize}
	\item $\xb$ splits if and only if 
	$\mathrm{He}_f(x,y) = -4ab \cdot (6xy)^2$ is a square of a quadratic form over $\F_p$.
\end{itemize}

Next we consider the case 
\begin{align*}
	f(x,y) &= (a + \sqrt{\beta} b)(x + \sqrt{\beta} y)^4 + (a - \sqrt{\beta} b)(x - \sqrt{\beta} y)^4 \\
	&= 2ax^4 + 8\beta bx^3y + 12\beta a x^2y^2  + 8 \beta^2 b xy^3 + 2\beta^2 a  y^4 \\
	&= 2a (x^2 + \beta y^2)^2 + 4 \beta b (x^2 + \beta y^2) \cdot 2xy + 2\beta a (2xy)^2.
\end{align*}
Its catalecticant matrix is
\[
	M_f = \begin{pmatrix}
		2a & 2\beta b & 2\beta a \\
		2\beta b & 2\beta a & 2\beta^2 b \\
		2\beta a & 2\beta^2 b & 2\beta^2 a
	\end{pmatrix},	
\]
and the defining equation of $\xb$ is
\[
	\begin{aligned}
	& \: \begin{pmatrix}
		t_0 & t_1 & t_2
	\end{pmatrix}
	\begin{pmatrix}
		2a & 2\beta b & 2\beta a \\
		2\beta b & 2\beta a & 2\beta^2 b \\
		2\beta a & 2\beta^2 b & 2\beta^2 a
	\end{pmatrix}
	\begin{pmatrix}
		t_0 \\ t_1 \\ t_2
	\end{pmatrix} 
	= 2a(t_0 + \beta t_2)^2 + 4 \beta b t_1(t_0 + \beta t_2) + 2\beta a t_1^2 = 0.
	\end{aligned}
\]
This can be considered as a quadratic form in two variables $t_0 + \beta t_2$ and $t_1$.
By this equation, we can see that
\begin{itemize}
	\item the singular point of $\xb$ is described as $t_0 + \beta t_2 = t_1 = 0$, that is, 
	$[\beta:0:-1]$, and
	\item $\xb$ splits if and only if 
	the discriminant (as a quadratic form in $t_0 + \beta t_2$ and $t_1$)
	$-4^2\beta (a^2-\beta b^2)$ is a square in $\F_p^{\times}$. 
	\end{itemize}

The covariant \yicmtwo{$\mathrm{He}_f$} in this case is
\begin{equation}\label{eq:he2}
	\mathrm{He}_f(x,y) = -4^2\beta (a^2 - \beta b^2) \cdot (6 (x^2 - \beta y^2))^2.
\end{equation}
Thus the latter statement is equivalent to
\begin{itemize}
	\item $\xb$ splits if and only if $\mathrm{He}_f(x,y)
	 = -4^2 \beta (a^2 - \beta b^2) \cdot (6 (x^2 - \beta y^2))^2$ is 
	a square of a quadratic form over $\F_p$.
\end{itemize}

By the above calculation, we have:

\begin{lem}
	For a semi-degenerate binary quartic form $f$, 
	the degenerate quadric $\xb$ splits into two $\F_p$-rational lines 
	if and only if 
	$\mathrm{He}_f(x,y)$ is the square of a binary quadratic form over $\F_p$.
\end{lem}

The last case we have to consider is when $f$ is degenerate, $J(f) = 0$, but $f$ is not of the splitting type $(1^4)$.
By Lemma \ref{lem:WaringDecomp}, 
this corresponds to the splitting type $(1^31)$ case, and by a linear change of coordinates, we only have to treat $f = x^3y$.
The corresponding quadratic form is
\(
	t_0t_1,
\)
a split union of two lines.

We count the $\F_p$-rational points in each case as follows:

\begin{itemize}
	\item {Generic case: $J(f) \neq 0.$}
	Then $\xb$ is \yicmtwo{a smooth conic}. As it has at least one $\F_p$-rational point, for example by the Chevalley-Warning theorem,
	it is isomorphic to $\Projsp^1$.
	Hence $\#\xb(\F_p) = p+1$.
	\item {$J(f)=0$ and split case:}
	Two lines intersect at a point in $\Projsp^2$. 
	Hence $\#\xb(\F_p) = 2(p+1)-1 = 2p+1$.
	\item {$J(f)=0$ and non-split case:}
	The only rational point is the intersection point.
	Hence $\#\xb(\F_p) = 1$.
	\item {$(1^4)$-case:}
	As the singular locus is $1$-dimensional, the quadric is a double line.
	Hence $\#\xb(\F_p) = p+1$.
\end{itemize}

Hence we have the following proposition:
\begin{prop}
	We have
	\[
\#\xb(\F_p)=		\begin{cases}
			p+1 & \mbox{$J(f) \neq 0$\yicmtwo{,} or $f$ is of the type $(1^4)$,} \\
			2p+1 & \mbox{semi-degenerate and $\mathrm{He}_f(x,y)$ is a square, or $f$ is of the type $(1^31)$,} \\
			1 & \mbox{semi-degenerate and $\mathrm{He}_f(x,y)$ is not a square.}
		\end{cases}
	\]
\end{prop}

\subsection{Counting points on the $(1^21^2)$-scheme}
We finally count $\#\xc(\F_p)$ in this subsection.
Let us consider $\xc \subset \Projsp^1 \times \Projsp^1$.
The required equality is represented in terms of coordinates as
\begin{equation*}
	\begin{split}
	\left(a_0s_0^2 + \frac{1}{2}a_1s_0s_1 + \frac{1}{6}a_2s_1^2\right)t_0^2 
	&+ \left(\frac{1}{2}a_1s_0^2 + \frac{2}{3}a_2s_0s_1 + \frac{1}{2}a_3s_1^2\right) t_0t_1\\ 
	&+ \left(\frac{1}{6}a_2s_0^2 + \frac{1}{2}a_3s_0s_1 + a_4s_1^2\right)t_1^2 = 0.
	\end{split}
\end{equation*}
Equivalently, 
\begin{equation}\label{case3surf}
	\begin{pmatrix}
		s_0^2 & 2s_0s_1 & s_1^2
	\end{pmatrix} 
	M_f
	\begin{pmatrix}
		t_0^2 \\ 2t_0t_1 \\ t_1^2
	\end{pmatrix} = 0.
\end{equation}
We denote the $(2,2)$-form in the \yicmtwo{left} hand side \yicmtwo{of \eqref{case3surf}} by $c_f=c_f(s_0,s_1;t_0,t_1)$.
According to \cite[Section 6.1]{BH},
this $(2,2)$-form defines a genus one curve if its discriminant does not vanish.
Here, by definition, the discriminant of a $(2,2)$-form 
\begin{equation}\label{eq:22}
	c(s_0, s_1;t_0, t_1) = q_0(s_0,s_1)t_0^2 + q_1(s_0,s_1)t_0t_1 + q_2(s_0, s_1)t_1^2
\end{equation} 
is the discriminant of the binary quartic form
$
	H_c(s_0, s_1) = q_1^2 - 4 q_0q_2.
$
In our case, we have 
\[
	H_{c_f}(s_0, s_1) = \frac{1}{36} \mathrm{He}_f (s_0, s_1),
\]
where $\mathrm{He}_f(x,y)$ is the covariant of $f$ defined in \eqref{eq:def_he}.
As stated in \cite[Proposition 5]{Cr},
the discriminant of $\mathrm{He}_f$ is equal to $2^{12}3^6J(f)^2\Disc(f)$
(here, the covariant $g_4$ in \cite{Cr} is equal to $1/3 \cdot \mathrm{He}_f$).
Hence, if $J(f) \neq 0$ and $\Disc(f) \neq 0$, 
the equation \eqref{case3surf} defines a genus one curve in $\Projsp^1 \times \Projsp^1$.

\subsubsection{Generic case.}
Let us assume that $J(f) \neq 0$ and $\Disc(f) \neq 0$.
We refer to this as the generic case.
By the Hasse--Weil bound $\xc$ has at least one rational point, 
and hence it is isomorphic to its Jacobian variety $\Jac(\xc)$ as curves over $\F_p$.
By Bhargava and Ho \cite[Section 6.1 (61)]{BH}, $\Jac(\xc)$ is given by an affine equation
\[
	y^2 + 216\delta_3(c_f)y = x^3 + 9\delta_2(c_f)x^2 + 27(\delta_2(c_f)^2 - \delta_4(c_f))x.
\]
Here, for a $(2,2)$-form $c=c(s_0, s_1;t_0, t_1)$ as in \eqref{eq:22}
with
\begin{align*}
	q_0(s_0, s_1) &= a_{00}s_0^2 + a_{01}s_0s_1 + a_{02}s_1^2, \\
	q_1(s_0, s_1) &= a_{10}s_0^2 + a_{11}s_0s_1 + a_{12}s_1^2, \\
	q_2(s_0, s_1) &= a_{20}s_0^2 + a_{21}s_0s_1 + a_{22}s_1^2, 
\end{align*}
its relative invariants $\delta_i(c) \; (i = 2, 3, 4)$ are defined by
\begin{align*}
	\delta_2(c) &= a_{11}^2 - 4a_{10}a_{12} + 8a_{02}a_{20} - 4a_{01}a_{21} + 8a_{00}a_{22}, \\
	\delta_3(c) &= -\det \left| \begin{matrix}
		a_{00} & a_{01} & a_{02} \\
		a_{10} & a_{11} & a_{12} \\
		a_{20} & a_{21} & a_{22}
	\end{matrix} \right|, \\
	\delta_4(c) &= I(H_c).
\end{align*}
In our case, 
\begin{align*}
	q_0(s_0, s_1) = \frac{1}{12} f_{xx}(s_0, s_1), \quad
	q_1(s_0, s_1) = \frac{1}{6} f_{xy}(s_0, s_1), \quad
	q_2(s_0, s_1) = \frac{1}{12} f_{yy}(s_0, s_1)
\end{align*}
and we have
\begin{align*}
	\delta_2(c_f) = \frac{2}{3}I(f), \quad
	\delta_3(c_f) = -\frac{1}{108}J(f), \quad
	\delta_4(c_f) = \frac{1}{9}I(f)^2.
\end{align*}
Thus $\Jac(\xc)$ is defined by
\[
	y^2 - 2J(f)y = x^3 + 6I(f)x^2 + 9I(f)^2x.
\]
We immediately see that this is actually isomorphic to 
\[
	E'_f \colon y^2 = x^3 -3I(f)x^2 + J(f)^2.
\]
Hence we have
\begin{prop}
	When $J(f) \neq 0$ and $\Disc(f) \neq 0$,
	we have
	\[
		\#\xc(\F_p) = \# E'_f(\F_p).
	\]
\end{prop}

\subsubsection{Semi-degenerate cases}
Next we treat the case when $J(f) = 0$ and $\Disc(f) \neq 0$ (the semi-degenerate case).
Recall that in this case, over $\overline{\F}_p$, we can write
\[
	f(x,y) = l_0(x,y)^4 + l_1(x,y)^4.
\]
Since we have
\[
	[f(x,y), (ax + by)^4] = f(a,b)
\]
for a binary quartic form $f(x,y)$, 
we obtain for each  linear form $l(x,y) = ax + by$ that
\[
		[(s_0 x + s_1y)^2 (t_0x + t_1y)^2, l(x,y)^4]
= l(s_0, s_1)^2 l(t_0, t_1)^2.
\]
Applying this calculation for $l_0$ and $l_1$, we find the defining equation of $\xc$ is
{\small
\begin{align}
	c(s; t) &= l_0(s_0, s_1)^2 l_0(t_0,t_1)^2 + l_1(s_0, s_1)^2 l_1(t_0,t_1)^2 \label{eq:c_eq_3} \\
	&= ( l_0(s_0, s_1) l_0(t_0,t_1) - \sqrt{-1} l_1(s_0, s_1) l_1(t_0,t_1)) 
	( l_0(s_0, s_1) l_0(t_0,t_1) + \sqrt{-1}l_1(s_0, s_1) l_1(t_0,t_1)). \nonumber
\end{align}
}
Hence geometrically $\xc$ consists of two distinct  $(1,1)$-divisors defined by $l_0(s_0, s_1) l_0(t_0,t_1) \pm\sqrt{-1} l_1(s_0, s_1) l_1(t_0,t_1)$.
By a linear change of coordinates over $\overline{\F}_p$, these divisors are rewritten as $s_0t_0 - s_1t_1$ and $s_0t_0 + s_1t_1$,
and we see that these divisors are irreducible and smooth. 
There are two kinds of splitting we have to consider:

\begin{enumerate}[{(i)}]
	\item Do two irreducible divisors split over $\F_p$?
	\item Do the two intersection points of these divisors split over $\F_p$?
\end{enumerate}

For the first question, similarly to $\xb$, it is enough to consider two cases 
\[
	f(x,y) = ax^4 + by^4
\]
and 
\[
	f(x,y) = (a + \sqrt{\beta} b)(x + \sqrt{\beta} y)^4 + (a - \sqrt{\beta} b)(x - \sqrt{\beta} y)^4.
\]
First we assume that \(f(x,y) = ax^4 + by^4.\)
Then we have
\[
		c(s; t) = as_0^2t_0^2 + bs_1^2t_1^2 
		= a \left( s_0t_0 + \frac{\sqrt{-ab}}{a} s_1t_1 \right) 
		\left( s_0t_0 - \frac{\sqrt{-ab}}{a} s_1t_1 \right)
\]
and it splits if and only if $-ab \in \F_p^{\times 2}$.
As we computed in \eqref{eq:he_def}, we have
$\mathrm{He}_f(x,y) = -4ab \cdot (6xy)^2$. Hence we have
\begin{itemize}
	\item For \yicmtwo{the case $f(x,y) = ax^4 + by^4$}, 
	the defining equation $c(s,t)$ of $\xc$ splits into two bilinear forms over $\F_p$
	if and only if $\mathrm{He}_f(x,y)$ is a square of a quadratic form over $\F_p$.
\end{itemize}
Next we assume that 
\[
	\begin{aligned}
		f(x,y) &= (a + \sqrt{\beta} b)(x + \sqrt{\beta} y)^4 + (a - \sqrt{\beta} b)(x - \sqrt{\beta} y)^4 \\
		&= 2a (x^2 + \beta y^2)^2 + 4 \beta b (x^2 + \beta y^2) \cdot 2xy + 2\beta a (2xy)^2.
	\end{aligned}
\]
Then using \eqref{eq:c_eq_3}, we have 
\[
	\begin{aligned}
	c(s; t) &= (a + \sqrt{\beta} b)(s_0 + \sqrt{\beta} s_1)^2(t_0 + \sqrt{\beta} t_1)^2 
	+ (a - \sqrt{\beta} b)(s_0 - \sqrt{\beta} s_1)^2(t_0 - \sqrt{\beta} t_1)^2 \\
	&= 2a (s_0t_0 + \beta s_1t_1)^2 + 4 \beta b (s_0t_0 + \beta s_1t_1) (s_1 t_0 + s_0t_1)
	+ 2 \beta a(s_1t_0 + s_0t_1)^2.
	\end{aligned}
\]
By the last expression, we find that as a quadratic form 
in $s_0t_0 + \beta s_1t_1$ and $s_1t_0 + s_0t_1$, 
$c(s;t)$ splits if and only if {$-4^2\beta (a^2-\beta b^2) \in \F_p^{\times 2}$}.
By \eqref{eq:he2}, again we have
\begin{itemize}
	\item For \yicmtwo{the case $f(x,y) = (a + \sqrt{\beta} b)(x + \sqrt{\beta} y)^4 + (a - \sqrt{\beta} b)(x - \sqrt{\beta} y)^4$}, 
	the defining equation $c(s,t)$ of $\xc$ splits into two bilinear forms over $\F_p$
	if and only if $\mathrm{He}_f(x,y)$ is a square of a quadratic form over $\F_p$.
\end{itemize}

Next we consider the second question.
For the case $f(x,y) = ax^4 + by^4$, 
\yicmtwo{regardless of whether $-ab \in \F_p^{\times 2}$ or not,} 
the intersection points are defined by the equations
\[
	s_0t_0 = s_1t_1 = 0.
\]
We easily find that the intersection points are the two points $([1:0], [0:1])$ and $([0:1], [1:0]) \in \Projsp^1 \times \Projsp^1$.
For the case $f(x,y) = (a + \sqrt{\beta} b)(x + \sqrt{\beta} y)^4 + (a - \sqrt{\beta} b)(x - \sqrt{\beta} y)^4$, 
the defining equations of the intersection points of the two divisors are written as
\[
	 s_0t_0 + \beta s_1t_1 = s_0t_1 + s_1t_0 = 0 
	\iff \begin{pmatrix} s_0 & \beta s_1 \\ s_1 & s_0 \end{pmatrix}  \begin{pmatrix} t_0 \\ t_1 \end{pmatrix} = \begin{pmatrix} 0 \\ 0 \end{pmatrix}.
\]
Since $\det  \begin{pmatrix} s_0 & \beta s_1 \\ s_1 & s_0 \end{pmatrix}  = s_0^2 - \beta s_1^2 = 0$ has no non-trivial solutions over $\F_p$, 
we find that there are no $\F_p$-rational intersection points of the two (1,1)-divisors.
Concluding then, the number of $\F_p$-rational intersection points of two divisors is two or zero,
according as whether $l_0$ and $l_1$ are defined over $\F_p$ or not.

Thus in the semi-degenerate case, we divide it into four subcases:
\begin{prop}\label{pp:semideg}
For semi-degenerate cases, we calculate that
\begin{enumerate}[{(i)}]
	\item Two split divisors intersecting in two rational points: $\#\xc(\F_p) = 2p$.
	\item Two split divisors intersecting in two non-rational points: $\#\xc(\F_p) = 2p+2$.
	\item Two non-split divisors intersecting in two rational points: $\#\xc(\F_p) = 2$.
	\item Two non-split divisors intersecting in two non-rational points: $\#\xc(\F_p) = 0$.
\end{enumerate}
They exactly correspond to 
\begin{enumerate}[{(i)}]
	\item $l_0$ and $l_1$ are defined over $\F_p$ 
	and $\mathrm{He}_f(x,y)$ is a square of a quadratic form over $\F_p$.
	\item $l_0$ and $l_1$ are not defined over $\F_p$ 
	and $\mathrm{He}_f(x,y)$ is a square of a quadratic form over $\F_p$.
	\item $l_0$ and $l_1$ are defined over $\F_p$ 
	and $\mathrm{He}_f(x,y)$ is not a square of a quadratic form over $\F_p$.
	\item $l_0$ and $l_1$ are not defined over $\F_p$ 
	and $\mathrm{He}_f(x,y)$ is not a square of a quadratic form over $\F_p$.
\end{enumerate}
\end{prop}

To relate this proposition to Theorem \ref{thm:ff-disc},
here we note the relation between these conditions and the invariant $I(f)$.
First we consider the case $f = ax^4 + by^4$ corresponding to the cases (i) and (iii), 
or equivalently, where $l_0$ and $l_1$ are defined over $\F_p$.  In this case, we compute
\[
	I(f) = 12ab.
\]
Hence, $\mathrm{He}_f(x,y) = -4ab \cdot (6xy)^2$ is a square of a quadratic form over $\F_p$
if and only if $\left( \frac{-3I(f)}{p}\right) = 1$.

In the remaining case $f = (a + \sqrt{\beta} b)(x + \sqrt{\beta}y)^4 + (a - \sqrt{\beta} b)(x - \sqrt{\beta}y)^4$, 
corresponding to the case (ii) and (iv), we have
\[
	I(f) = 192 \beta^2 (a^2 - \beta b^2).
\]
Thus $\mathrm{He}_f(x,y)$ as computed in \eqref{eq:he2} is a square of a quadratic form over $\F_p$
if and only if $\left( \frac{-3I(f)}{p}\right) = -1$.
In summary, we have the following:

\begin{lem}\label{lm:IvDescription}
	For a semidegenerate $f \in V$, 
	$\left( \frac{-3I(f)}{p}\right) = 1$ if and only if case (i) or (iv) holds
	in Proposition \ref{pp:semideg}.
\end{lem}

\subsubsection{$(1^211)$- and $(1^22)$-cases}
We may assume that $f(x,y) = x^2q(x,y)$ for a quadratic form $q(x,y)$ over $\F_p$ with
$\disc(q) \neq 0$ and $q(0,1) \neq 0$.
Then we have
\[
	c(s;t) = a_0s_0^2t_0^2 + \frac{1}{2}a_1s_0t_0(s_1t_0 + s_0t_1) + \frac{1}{6}a_2(s_0^2t_1^2 + 4s_0s_1t_0t_1 + s_1^2t_0^2).
\]
The intersection with the $(1,0)$-divisor $s_0 = 0$ 
or the $(0,1)$-divisor $t_0 = 0$ is a singular point $([0:1], [0:1])$.
Localizing with $s_0 \neq 0$ and $t_0 \neq 0$, we have a quadratic polynomial
\[
	\overline{c}(s;t) = a_0 + \frac{1}{2}a_1(s_1 + t_1) + \frac{1}{6}a_2(t_1^2 + 4s_1t_1 + s_1^2).
\]
Homogenizing it, we have a ternary quadratic form
\[
	\widetilde{c}(s,t,u) = a_0u^2 + \frac{1}{2}a_1(s_1 + t_1)u + \frac{1}{6}a_2(t_1^2 + 4s_1t_1 + s_1^2).
\]
Its discriminant is
\[
	\disc(\widetilde{c}) = \frac{1}{48}a_2\disc(q),
\]
and by assumption this quadratic form is non-singular.
Thus it defines a non-singular quadric in $\Projsp^2$, and there are
$p+1$ rational points.
The two points where the quadric meets the line $u=0$ are
\[
	[s_1:t_1:u]=[1:-2+\sqrt{3}:0], [1:-2-\sqrt{3}:0],
\]
and the number of $\F_p$-rational points is $1+\left( \frac{3}{p} \right)$.
Hence we have
\[
	\#\xc(\F_p) = 1 + (p+1) - \left(1+\left( \frac{3}{p} \right)\right) 	= p + 1 - \left( \frac{3}{p} \right).
\]
\begin{prop}
	For a binary quartic $f$ of $(1^211)$- or $(1^22)$-type, we have
	\[
		\#\xc(\F_p) = p + 1 - \left( \frac{3}{p} \right).
	\]
\end{prop}

\subsubsection{$(1^21^2)$- and $(2^2)$-cases}
According to the splitting type, 
we may assume that $f(x,y) = x^2y^2$ or $(x^2-\beta y^2)^2$ respectively
(where $\beta \in \F_p^\times$ is a quadratic non-residue).
For $f=x^2y^2$, we have
\begin{align*}
	c(s;t) &= \frac{1}{6}(s_0^2t_1^2 + 4s_0s_1t_0t_1 + s_1^2t_0^2) \\
	&= \frac{1}{6}\left( 
	s_0t_0 + ({2+\sqrt{3}})s_1t_1
	\right)
	\left( 
	s_0t_0 + ({2-\sqrt{3}})s_1t_1
	\right).
\end{align*}
Both $s_0t_0 + ({2+\sqrt{3}})s_1t_1$
and $s_0t_0 - ({2+\sqrt{3}})s_1t_1$ define 
an irreducible divisor.
Hence $\xc$ is a split union of two irreducible $(1,1)$-divisors if $\left( \frac{3}{p} \right) = 1$,
or a non-split union if $\left( \frac{3}{p} \right)=-1$.
Their intersection points are two $\F_p$-rational points $([1:0], [0:1])$ and $([0:1], [1:0]) \in \Projsp^1 \times \Projsp^1$.

For $f = (x^2-\beta y^2)^2$, we have
\begin{align*}
	c(s;t) &= (s_0t_0-\beta s_1t_1)^2 - \frac{\beta}{3}(s_0t_1 - s_1t_0)^2.
\end{align*}
It is a split union of two irreducible $(1,1)$-divisors if $\left( \frac{3}{p} \right) = -1$,
or a non-split union if $\left( \frac{3}{p} \right)=1$.
These divisors intersect on two $\overline{\F}_p$-rational points, but they are not $\F_p$-rational.

\begin{prop}
	For a binary quartic $f$ of $(1^21^2)$-type, we have
\[
		\#\xc(\F_p) = (p+1) + \left( \frac{3}{p} \right)(p-1).
\]
	For a binary quartic $f$ of $(2^2)$-type, we have
\[
		\#\xc(\F_p) = (p+1) - \left( \frac{3}{p} \right)(p+1).
\]
\end{prop}

\subsubsection{$(1^31)$-case}
We only have to consider the $f = x^3y$ case.
The corresponding $(2,2)$-form is 
\[
	\frac{1}{2}s_0^2t_0t_1 + \frac{1}{2}s_0s_1t_0^2 = \frac{1}{2}s_0t_0(s_0t_1+s_1t_0).
\]
This is the union of three components: two lines $s_0 = 0$ and $t_0 = 0$, and
the graph of an isomorphism $[s_0:s_1] \mapsto [-s_0:s_1]$ of $\Projsp^1$: $s_0t_1 + s_1t_0=0$.
They are all isomorphic to $\Projsp^1$,
and intersect only at $([0:1], [0:1])$.
Thus we have $\#\xc(\F_p)= 3((p+1) - 1) + 1 = 3p+1$.

\subsubsection{$(1^4)$-case}
We only have to consider the $f = x^4$ case.
The corresponding $(2,2)$-form is $s_0^2t_0^2$,
and thus we have $\# \xc(\F_p) = 2(p+1) - 1 = 2p+1$.

\subsection{Proof of Theorem \ref{thm:ff-disc}}\label{subsec:proof-expsum}
Now we are ready to prove Theorem \ref{thm:ff-disc}.
Recall from \eqref{eq:expo_breakdown} and \eqref{eq:1_breakdown} that 
\begin{align*}
	p^5 \widehat{\Phi_p}(f) &= 1 + p(\#\xa(\F_p) + \#\xb(\F_p) - \#\xc(\F_p)) - (p^3 + 2p^2 + p + 1) \\
	&= p(\#\xa(\F_p) + \#\xb(\F_p) - \#\xc(\F_p) - (p+1)^2).
\end{align*}

We distinguish the following cases:
\begin{itemize}
\item $J(f) \neq 0$ and $\Disc(f) \neq 0$, the nondegenerate case;
\item $J(f) = 0$ and $\Disc(f) \neq 0$, the semi-degenerate case, divided into four subcases in Proposition \ref{pp:semideg};
\item $J(f) \neq 0$ and $\Disc(f) = 0$, corresponding to the splitting types $(1^2 11)$, $(1^2 2)$, $(1^2 1^2)$, $(2^2)$ by Lemma \ref{lem:WaringDecomp};
\item $J(f) = 0$ and $\Disc(f) = 0$, corresponding to the splitting types $(1^3 1)$ and $(1^4)$ by the same lemma.
\end{itemize}

In summary, we obtain the following table.
We write $\left( \frac{3}{p} \right)=\chi_{12}(p)$
since it is the unique primitive Dirichlet character modulo $12$.

\begin{center} 
\begin{table}[h]
\begin{tabular}{|c||c|c|c|c|}
	\hline
	$f$ & $\#\xa(\F_p)$ & $\#\xb(\F_p)$ & $\#\xc(\F_p)$ & $p^5 \widehat{\Phi_p}(f)$ \\
	\hline
	\hline
	$J(f) \neq 0, \Disc(f) \neq 0$	& $(p+1)^2$ & $p+1$ & $\#E'_f(\F_p)$ & $pa(E'_f)$ \\
	\hline
	semideg., (i) 	& $(p+1)^2$ & $2p+1$ & $2p$ & $p$ \\
	\hline
	semideg., (ii)	& $(p+1)^2$ & $2p+1$ & $2p+2$ & $-p$ \\
	\hline
	semideg., (iii)	& $(p+1)^2$ & $1$ & $2$ & $-p$ \\
	\hline
	semideg., (iv)	& $(p+1)^2$ & $1$ & $0$ & $p$ \\
	\hline
	$(1^211)$ 	& $(p+1)^2$ & $p+1$ & $p + 1 - \chi_{12}(p)$ & $\chi_{12}(p)p$ \\
	\hline
	$(1^22)$ 	& $(p+1)^2$ & $p+1$ & $p + 1 - \chi_{12}(p)$ & $\chi_{12}(p)p$ \\
	\hline
	$(1^21^2)$ 	& $(p+1)^2$ & $p+1$ & $(p+1)+\chi_{12}(p)(p-1)$ & $-\chi_{12}(p)p(p-1)$ \\
	\hline
	$(2^2)$ 	& $(p+1)^2$ & $p+1$ & $(p+1)-\chi_{12}(p)(p+1)$ & $\chi_{12}(p)p(p+1)$ \\
	\hline
	$(1^31)$ 	& $2p^2+2p+1$ & $2p+1$ & $3p+1$ & $p^2(p-1)$ \\
	\hline
	$(1^4)$ 	& $2p^2+2p+1$ & $p+1$ & $2p+1$ & $p^2(p-1)$\\
	\hline
\end{tabular} 
\end{table}
\end{center}

By Lemma \ref{lm:IvDescription},
any semi-degenerate $f$ is of case (i) or (iv) if and only if $\left(\frac{-3I(f)}{p}\right) = 1.$
Thus we have proved Theorem \ref{thm:ff-disc}.

\section{Box Estimate for the space of binary quartic forms}\label{sec:box_estimate}
To prove Theorems \ref{thm:ap-2selmer} and \ref{thm:ap-bqf},
we will use the framework developed in \cite{TT_leveldist}.
As this was designed to work 
essentially as a black box, we give only a brief summary. 

For squarefree $q$,
let $\Phi_q \colon V(\Z/q\Z)\rightarrow\{0,1\}$
be the indicator of the those $f\in V(\Z/q\Z)$ with $\disc(f)=0$. Then, by Poisson summation,
we have
\begin{equation}\label{eq:poisson1}
\sum_{f \in V(\Z)} \Phi_q(f) \phi(f X^{-1/6}) = \widehat{\Phi_q}(0) \widehat{\phi}(0) X^{5/6} + 
\sum_{0 \neq f \in V^*(\Z)} \widehat{\Phi_q}(f) \widehat{\phi} \bigg( \frac{f X^{1/6}}{q} \bigg),
\end{equation}
where we regard the sum over $f \neq 0$ as an error term, which 
if $\phi$ is a Schwartz function is essentially supported on a box of side length $q X^{-1/6}$. 
If we can bound this error by $o(X^{5/6})$, when summed over $q \leq X^{\alpha}$, then we obtain a {\itshape level of
distribution} of $\alpha$ for the function $\Phi_q$. 

Following \cite[\S 3.1]{TT_leveldist}, we note that the bilinear form \eqref{eq:bilinear} defines 
an injection $\rho \ : \ V^\ast(\Z) \hookrightarrow V(\Z)$ whose image contains $12 V(\Z)$. We did not compute 
$\widehat{\Phi_2}$ or $\widehat{\Phi_3}$, but note that they are trivially bounded above by $1$. Therefore,
by abuse of notation we {\itshape define} $\widehat{\Phi_q}(f) := \widehat{\Phi_\frac{q}{(q, 6)}}(f)$ for all squarefree $q$ in what follows,
and use this definition to extend $\widehat{\Phi_q}(f)$ from $\textnormal{Im}(\rho)$ to all of $V(\Z)$.

This was all developed more formally in \cite{TT_leveldist}, and we will prove the following statement,
designed to satisfy the hypothesis of \cite[Proposition 10]{TT_leveldist}. In what follows, we write
$B\subset V(\R)$ for the box of side length $2$ centered in the origin,
so that $rB\subset V(\R)$ is the set of binary quartic forms whose
absolute values of coefficients are all bounded by $r$.

\begin{cor}\label{cor:box-estimate}
Let $\alpha$ be an arbitrary real number with $\alpha<1/3$.
There exist $c<5/6$ and $\eta>0$ such that
for all $Q<X^\alpha$, with $r=QX^{\eta-1/6}$, we have
\[
X^{5/6}\sum_{q \in [Q, 2Q]} 
\sum_{\substack{0 \neq f \in V(\Z)\cap rB}}
|\widehat{\Phi_q}(f)|
\ll
X^c,
\]
where the implied constant may depend on
$\alpha$, $c$ and $\eta$.
\end{cor}

The above is deduced immediately from the following bound:
\begin{thm}\label{thm:box-estimate}
Let $\epsilon>0$ be arbitrary. For $Q>r$,
\begin{equation}\label{eq:box-estimate}
\sum_{q \in [Q, 2Q]} 
\sum_{\substack{0 \neq f \in V(\Z)\cap rB}}
|\widehat{\Phi_q}(f)|
\ll
\left(
\frac{r^2}{Q}+\frac{r^4}{Q^2}+\frac{r^5}{Q^{5/2}}
\right)
Q^\epsilon.
\end{equation}
Here, $q$ runs through all squarefree integers in the range $[Q,2Q]$.
\end{thm}

Our application of Corollary \ref{cor:box-estimate} will be discussed
more in detail in the next section.
For the rest of this section
we focus on the proof of Theorem \ref{thm:box-estimate}.

Let $\fx\subset V$ be the closure of the set of binary quartic forms
which have either triple roots or distinct double roots in $\mathbb P^1$.
This $\fx$ is a $3$ dimensional subvariety of $V$ defined over $\Q$.
To prove Theorem \ref{thm:box-estimate},
we first want to bound the cardinality of
the set $V(\Z)\cap rB\cap \fx(\Q)$.
Since $\fx$ is of dimension $3$, by general theory we have
$|V(\Z)\cap rB\cap \fx(\Q)|\ll r^3$.
For our particular $\fx$, we have 
the following stronger bound:
\begin{lem}\label{lem:integral-singular-pts}
We have
\begin{equation}
|V(\Z)\cap rB\cap \fx(\Q)|\ll r^2.
\end{equation}
\end{lem}
\begin{proof}
If $f\in V(\Z)\cap \fx(\Q)$,
either $f=(ax+by)^3(cx+dy)$ for some integers $a,b,c,d$
or $f=t(ax^2+bxy+cy^2)^2$ for some integers $t,a,b,c$.

Suppose $f=(ax+by)^3(cx+dy)$.
If $a=0$, then $f=(0,0,0,*,*)$ and the number of
such $f\in rB$ is $O(r^2)$. The same holds for $b=0$.
If $ab\neq0$, since $f=(a^3c,*,*,*,b^3d)$,
it is enough to bound the number of $(a,b,c,d)$ with
$|a^3c|\leq r$ and $|b^3d|\leq r$, which is
\[
\left(
\sum_{0\neq|a|\leq r}\sum_{|c|\leq r/|a|^3}1
\right)^2
=O(r^2).
\]

Suppose $f=t(ax^2+bxy+cy^2)^2$.
We obviously can assume $t\neq0$. If some coefficient of $ax^2 + bxy + cy^2$ has absolute value greater than $H$, 
then at least one coefficient of $(ax^2 + bxy + cy^2)^2$ has absolute value greater than $H^2 / 2$. Taking $|t|H^2 / 2 = r$,
we see that 
the number of possible such $(t,a,b,c)$ is 
\[
\ll 
\sum_{0\neq|t|\leq r}
\bigg( \frac{r}{t} \bigg)^{3/2} \ll r^{3/2}.
\]
Thus we have the desired result.
\end{proof}

We now prove Theorem \ref{thm:box-estimate}.

\begin{proof}[Proof of Theorem \ref{thm:box-estimate}]
We choose and fix an integral model of $\fx$,
and use the same notation $\fx$ for the integral model.
Then except for a finite number of primes $p$,
$\fx(\F_p)$ is the set of binary quartic forms
whose splitting type is either
$(0), (1^4), (1^31), (1^21^2)$ or $(2^2)$.
Thus Theorem \ref{thm:ff-disc} implies
\begin{equation}\label{eq:exp-sum-bound}
|\widehat\Phi_p(f)|
\ll
\begin{cases}
p^{-1}	&	f\bmod p=0,\\
p^{-2}	&	f\bmod p\in \fx(\F_p)\setminus\{0\},\\
p^{-7/2}	& f\bmod p\notin \fx(\F_p).
\end{cases}
\end{equation}
For each pair $(f,q)$ in the sum of \eqref{eq:box-estimate},
we consider the decomposition $q=q_0q_3q_5$ of $q$
where
$f\bmod{q_0}=0$,
$f\bmod{q_3}\in \fx(\Z/q_3\Z)\setminus\{0\}$
and
$f\bmod{q_5}\in V(\Z/q_5\Z)\setminus \fx(\Z/q_5\Z)$.
Note that $q_0\leq r$ since $f\in rB$.
Then by \eqref{eq:exp-sum-bound},
\[
|\widehat{\Phi_q}(f)|\ll q_0^{-1}q_3^{-2}q_5^{-7/2}Q^\epsilon.
\]
We split the double sum in the left hand side
of \eqref{eq:box-estimate} into three parts
$S_1$, $S_2$ and $S_3$, where:
\begin{itemize}
\item
$S_1$ is the sum over pairs $(q,f)$ with $f\in \fx(\Q)$;
\item
$S_2$ is the sum over pairs $(q,f)$ 
with $f\notin \fx(\Q)$ and $q_0q_3>r$;
\item
$S_3$ is the sum over pairs $(q,f)$ 
with $f\notin \fx(\Q)$ and $q_0q_3\leq r$.
\end{itemize}
(The three terms in the right hand side of
\eqref{eq:box-estimate} respectively correspond
to $S_1$, $S_2$ and $S_3$.)

We first consider $S_1$.
For any such $f$, $q_5=1$ and we have
\[
S_1\ll
Q^\epsilon
\sum_{q_0\leq r}
\sum_{q_3\sim\frac{Q}{q_0}}
\sum_{f\in q_0V_\Z\cap rB\cap \fx(\Q)}q_0^{-1}q_3^{-2},
\]
where $q_3\sim\frac{Q}{q_0}$
means $q_3\in[\frac{Q}{q_0},\frac{2Q}{q_0}]$.
By Lemma \ref{lem:integral-singular-pts},
we have
\[
|q_0V_\Z\cap rB\cap \fx(\Q)|
=
|V_\Z\cap \frac{r}{q_0}B\cap \fx(\Q)|
\ll\frac{r^2}{q_0^2}.
\]
Therefore
\[
S_1\ll
r^2Q^\epsilon
\sum_{q_0\leq r}
\frac1{q_0^3}
\sum_{q_3\sim\frac{Q}{q_0}}
\frac{1}{q_3^2}
\ll
r^2Q^\epsilon
\sum_{q_0\leq r}
\frac1{q_0^3}\cdot \frac{q_0}Q
\ll
\frac{r^2Q^\epsilon}{Q}.
\]

We next consider $S_3$.
By definition,
\[
S_3\ll
Q^\epsilon
\sum_{q_0q_3\leq r}
\sum_{q_5\sim\frac{Q}{q_0q_3}}
\frac{1}{q_0q_3^2q_5^{7/2}}
\cdot
\#\left\{f\in V(\Z)\cap rB
\ 
\vrule
\ 
\begin{array}{l}
f\bmod p=0\  (\forall p\mid q_0)\\
f\bmod p\in \fx(\F_p)\ (\forall p\mid q_3)\\
\end{array}
\right\}.
\]
The conditions on $f\bmod{q_0q_3}$
restrict $f$ to lie in a union of $O(q_3^{3+\epsilon})$
residue classes $\bmod{\ q_0q_3V(\Z)}$.
Since $q_0q_3\leq r$, each class contains
$\ll\big(\frac{r}{q_0q_3}\big)^5$ elements of $V(\Z)\cap rB$.
Thus we have
\begin{align*}
S_3
&\ll
Q^\epsilon
\sum_{q_0q_3\leq r}
\sum_{q_5\sim\frac{Q}{q_0q_3}}
\frac{1}{q_0q_3^2q_5^{7/2}}
\cdot q_3^3
\cdot
\left(
\frac{r}{q_0q_3}
\right)^5\\
&
=
r^5Q^\epsilon
\sum_{q_0q_3\leq r}
\frac{1}{q_0^6q_3^4}
\sum_{q_5\sim\frac{Q}{q_0q_3}}
\frac{1}{q_5^{7/2}}
\ll
r^5
Q^\epsilon
\sum_{q_0q_3\leq r}
\frac{1}{q_0^6q_3^4}
\left(
\frac{Q}{q_0q_3}
\right)^{-5/2}
\ll
\frac{r^5Q^\epsilon}{Q^{5/2}}.
\end{align*}

Finally, we study $S_2$.
We consider a further decomposition
$q_3=q_3'q_3''$, where 
$q_3'$ is the largest divisor of $q_3$ for which $q_0q_3'\leq r$.
Then $q_3''>1$ and $q_0q_3'p>r$ for any $p\mid q_3''$.
 We have
\begin{align*}
S_2
&
\ll
Q^\epsilon
\sum_{q_0q_3'\leq r}
\sum_{\substack{1<q_3''\leq\frac{2Q}{q_0q_3'}\\ (q_0q_3',q_3'')=1\\ \forall p\mid q_3'',q_0q_3'p>r}}
\sum_{\substack{f\in(V_\Z\cap rB)\setminus \fx(\Q)\\ f\bmod p=0\  (\forall p\mid q_0)\\
f\bmod p\in \fx(\F_p)\ (\forall p\mid q_3)}}
\sum_{q_5\sim\frac{Q}{q_0q_3'q_3''}}
\frac{1}{q_0(q_3' q_3'')^2q_5^{7/2}}.
\end{align*}
The innermost sum is $\ll(q_0/Q^2)\sum_{q_5}q_5^{-3/2}\ll q_0/Q^2$.
We divide the sum over $f\in V_\Z$ according to its residue class modulo
$q_0q_3'$, and have the bound
\begin{align*}
S_2
&\ll\frac{Q^\epsilon}{Q^2}
\sum_{q_0q_3'\leq r}q_0
\sum_{f_0 \bmod q_0 q_3'}
\#\left\{(f,q_3'')
\ 
\vrule
\ 
\begin{array}{l}
f\in(f_0+q_0q_3'V_\Z)\cap rB, f\notin \fx(\Q),\\
1<q_3''\text{ : squarefree}, (q_3'',q_0q_3')=1,\\
\forall p\mid q_3'', pq_0q_3'>r,\\
f\bmod p\in \fx(\F_p)\ (\forall p\mid q_3'')\\
\end{array}
\right\}. 
\end{align*}
By a variation \cite[(33)]{TT_leveldist} of the Ekedahl-Bhargava geometric sieve \cite{B_geosieve}, there are $\ll (r/q_0 q_3')^{4} r^{\epsilon}$ elements $f$ for which the inner count $\# (f, q_3'')$ is nonzero.
To bound the multiplicity, note that $\fx$ is given by a finite number of equations $h_1(f) = 0, \dots, h_k(f) = 0$. (In
this particular case we may take $k = 2$, $h_1$ to be the discriminant, and $h_2$ to be its derivative with respect to the first variable.)
Each pair $(f, q_3'')$ contributing to the sum satisfies
\[
h_i(f) \neq 0, \qquad h_i(f) \equiv 0 \pmod{q_3''}
\]
for some $i$. In particular, $q_3''$ must be a nontrivial divisor of one of the $h_i(f)$, and hence there are $\ll r^{\epsilon}$ options.
We conclude that 
\begin{align*}
S_2
&\ll
\frac{Q^\epsilon}{Q^2}
\sum_{q_0q_3'\leq r}q_0
\sum_{f_0}
\left( \frac{r}{q_0 q_3'} \right)^4 \\
& \ll \frac{Q^\epsilon}{Q^{2}}
\sum_{q_0q_3'\leq r}q_0
(q'_3)^{3 + \epsilon} \left( \frac{r}{q_0 q_3'} \right)^4 \ll 
\frac{r^4Q^\epsilon}{Q^{2}}.
\end{align*}
This finishes the proof.
\end{proof}

\section{Almost prime discriminants}\label{sec:ap}

In this section we 
prove Theorems \ref{thm:ap-2selmer} and \ref{thm:ap-bqf}
by applying Corollary \ref{cor:box-estimate}.
As we will see, Theorem \ref{thm:ap-bqf} is proved
in the process of proving Theorem \ref{thm:ap-2selmer}.

We begin by recalling the Birch--Swinnerton-Dyer parametrization \cite{BSD} 
of $2$-Selmer groups of elliptic curves in terms of binary quartic forms,
in the formulation of Bhargava
and Shankar \cite{BS2}.
Recall our definitions of the invariants
$I, J$ and $\disc$ from Section \ref{sec:background}.
For $f\in V(\Z)$ we define the {\itshape height}
of a binary quartic form to be $H(f) := \max(|I(f)^3|, J(f)^2/4)$.

We will need the notion of {\itshape solubility}
and local solubility from Bhargava and Shankar's work.
Following \cite[Section 3.1]{BS2}, we say that a binary quartic form
$f$ over a field $K$ is $K$-soluble if the equation $z^2 = f(x, y)$
has a solution with $x, y, z \in K$ and $(x, y) \neq (0, 0)$. Over $\Q$,
we further say that $f$ is locally soluble if
it is soluble over $\R$ and over every $p$-adic field $\Q_p$.

The parametrization of $2$-Selmer groups, as stated and applied by Bhargava and Shankar \cite[Theorem 3.5]{BS2} and originating in the work
of Birch and Swinnerton-Dyer \cite{BSD}, is:

\begin{thm}[\cite{BSD, BS2}]\label{thm:bsd-bs} Let 
$E : y^2 = x^3 - (I/48)x - J/1728$ be an elliptic curve over $\Q$.
Then there exists a bijection between elements in the $2$-Selmer group of $E$, 
and $\PGL_2(\Q)$-equivalence classes of locally soluble integral binary quartic forms having invariants
$I$ and $J$.

Further, the forms without a linear factor correspond to nontrivial elements of ${\rm Sel}_2(E)$.
\end{thm}

We now start the proof of Theorems \ref{thm:ap-2selmer} and \ref{thm:ap-bqf}.
Let
\[
f_0(x,y) :=-x^4-38x^3y-12x^2y^2-8xy^3\in V(\Z),
\]
\[
S :=\{g \cdot f\in V(\Z) : g \in \GL_2(\Z), f\in V(\Z),\ f \equiv f_0\pmod{3^3\cdot 2^{12}}\}.
\]
Then $S$ is defined by congruence conditions modulo
$3^3\cdot 2^{12}$ on $V(\Z)$. For $f \in S$, we have
\begin{equation}\label{eq:cong-for-S}
I(f)\equiv I(f_0) = - 3\cdot 2^8 \pmod{3^3\cdot 2^{12}}, \ \ 
J(f)\equiv J(f_0) = - 3^3 \cdot 2^{10} \pmod{3^3\cdot 2^{12}},
\end{equation}
and we check from \eqref{eq:def_disc} that $\Disc(f)/2^{20}$ is an integer coprime to $6$. The reason for our choice of $S$ will become apparent shortly.

Our key analytic result is the following:
\begin{thm}\label{thm:used_sieve}
For any $\alpha<1/3$, we have
\begin{equation}\label{eq:ap-S}
\#\left\{
f\in\GL_2(\Z)\backslash S
\ \vrule\ 
\begin{array}{l}
\text{$H(f)<X$}\\
p\mid \Disc(f)/2^{20}\Rightarrow p>X^{\alpha/4}\\
\Omega(\Disc(f)/2^{20})\leq 4\\
\text{$f$ is $\R$-soluble}\\
\text{$f$ is irreducible over $\Q$}\\
\text{$\Disc(f)/2^{20}$ is squarefree}
\end{array}
\right\}
\gg\frac{X^{5/6}}{\log X}.
\end{equation}
Further, all $f$ counted in \eqref{eq:ap-S} are locally soluble.
\end{thm}
\begin{proof}
This is essentially identical to  \cite[Proposition 15]{TT_leveldist},
applying Corollary \ref{cor:box-estimate}
in combination with the weighted sieve of
Richert \cite{richert} and Greaves \cite{greaves}
in the form \cite[Theorem 5]{TT_leveldist}.

We first show this lower bound when the bottom two conditions on the left of \eqref{eq:ap-S} are omitted.
We begin, as in \cite{TT_leveldist}, by replacing the sharp cutoff $H(f) < X$ by a nonzero smooth weighting function of the form
$\tilde\phi\colon f \mapsto \phi(X^{-1/6} f)$, where
$\tilde \phi$ is supported on a bounded set consisting exclusively
forms with height in $(0, 1)$, within one of the
sets $\mathcal{F} \cdot L^{(i)}$ ($i=0,1$, or $2+$)
described in \cite[Section 2.1]{BS2},
and satisfying $0 \leq \phi(f) \leq \frac{1}{8}$ for all $f$. 
As explained in \cite{BS2}, this construction guarantees that
each $\GL_2(\Z)$-orbit will be represented by at most eight points in the
support of $\tilde \phi$, and all such points are $\R$-soluble.
To count integer orbits only of $S$, we further let
$\Psi_S$ be the function on $V(\Z/3^32^{12}\Z)$ detecting $S$. 
Then
\[
\sum_{f\in S}\phi(X^{-1/6}f)=
\sum_{f\in V(\Z)}\Psi_S(f)\phi(X^{-1/6}f)
\]
is a smooth undercount for the $\R$-soluble
integral orbits $\GL_2(\Z)\backslash S$
with height up to $X$.

We now follow the argument in \cite{TT_leveldist} exactly,
using Poisson summation
and using Corollary \ref{cor:box-estimate} to bound the error terms, setting up an application of the weighted sieve.
For each squarefree integer $q$ coprime to $6$, Poisson summation yields
\begin{equation}\label{eq:smooth-count}
\sum_{f\in V(\Z)}\Phi_q(f)\Psi_S(f)\phi(X^{-1/6}f) = \widehat{\Phi_q}(0)\widehat{\Psi_S}(0)\widehat\phi(0) + E(X, q),
\end{equation}
for an error term $E(X, q)$ identical to that in \eqref{eq:poisson1}, but with the level of support expanded by a (harmless) factor of $3^32^{12}$.
Then $\widehat{\Phi_q}(0)$ is multiplicative in $q$, and
$\omega(p) :=\widehat{\Phi_p}(0)$
satisfies a one-sided linear sieve inequality
in the form \cite[(11)]{TT_leveldist}.
Then Corollary \ref{cor:box-estimate} asserts that
the hypothesis of \cite[Proposition 10]{TT_leveldist}
is satisfied, so that for any $\alpha<1/3$ we have
\[
\sum_{q<X^\alpha}|E(X,q)|\ll X^{5/6-\delta}
\]
for some $\delta>0$, where the sum is over squarefree $q$ coprime to $6$. Thus \cite[Theorem 5]{TT_leveldist},
with $t=4$, ensures the lower bound \eqref{eq:ap-S} (without the bottom two conditions on the left).

To complete the proof of \eqref{eq:ap-S},
we show that 
only a negligible number of integer orbits
are removed by the bottom two conditions on the left.
The number of forms which are reducible is
$\ll X^{4/6 + \epsilon}$, by
\cite[Lemma 2.3]{BS2}.
In our case the proof simplifies since
$\phi$ is compactly supported and thus
we are only counting points in a box of side length $O(X^{1/6})$.
Meanwhile, by a `tail estimate' of
Shankar, Shankar, and Wang \cite[Theorem 6.5]{SSW}, there are 
$\ll X^{5/6 - \alpha/4 + \epsilon} + X^{19/24}$ forms counted in \eqref{eq:ap-S} whose discriminant is divisible by $p^2$ for any $p > X^{\alpha/4}$.
Again this bound is negligible, 
allowing us to add the condition that $\Disc(f)/2^{20}$ is squarefree.

Finally, we check that all $f$ counted in \eqref{eq:ap-S} are $\Q_p$-soluble
for every prime $p$.
For $p$ odd, since $p^2 \nmid \Disc(f)$ this is implied by
\cite[Proposition 3.18]{BS2}. For $p = 2$,
we may assume $f\equiv f_0\pmod{2^{12}}$. Since $f_0(2,-1)=2^8$,
we have $f(2,-1)=2^8(1+2^4u)$ for $u\in\Z$,
which is a $2$-adic square as needed.
\end{proof}

This theorem includes Theorem \ref{thm:ap-bqf} as a special case.
We note that this proof shows that 
we may add finitely many suitable $\GL_2(\Z)$-invariant
congruence conditions
on the binary quartic forms in Theorem \ref{thm:ap-bqf},
and the same results still hold.

To prove Theorem \ref{thm:ap-2selmer},
we translate Theorem \ref{thm:used_sieve} 
into the language of $2$-Selmer elements of elliptic curves over $\Q$.
Let
\[
\mathcal E_\Z
:=\left\{
	(-48A,-1728B)\ \vrule
		\begin{array}{l}
			A,B\in V(\Z),
			p^4\mid A\Rightarrow p^6\nmid B\ (\forall p),\\
			-4A^3-27B^2\neq0
		\end{array}
\right\}.
\]
Note that $48=3\cdot2^4$ and $1728=3^3\cdot2^6$.
For each $\mathcal E_\Z\ni (I,J)=(-48A,-1728B)$,
we associate an elliptic curve
\[
E^{IJ}\colon y^2=x^3+Ax+B=x^3-\frac{I}{48}x-\frac{J}{1728}.
\]
This $(I,J)\mapsto E^{IJ}$ gives a bijection between $\mathcal E_\Z$
and the set of isomorphism classes of elliptic curves over $\Q$.
The (naive) height of this elliptic curve is defined (as in \cite{BS2}) by
\[
H(E^{IJ}):=\max\{4|A|^3,27B^2\}=\frac{1}{3^3\cdot 2^{10}}\max\{|I|^3,J^2/4\}
\]

For a non-singular binary quartic form $f\in V(\Q)$,
the elliptic curve $E_f$  associated by Theorem \ref{thm:bsd-bs} is defined by
\[
E_f\colon y^2=x^3-\frac{I(f)}{48}-\frac{J(f)}{1728}.
\]
We consider the discriminant and height of $E_f$
for $f$ we are counting. We have:
\begin{lem}\label{lem:E_f}
Let $f\in S$ and 
\yicmtwo{suppose} that $\Disc(f)/2^{20}$ is not divisible by $p^{12}$ for any prime $p$.
Then we have 
$H(E_f)=H(f)/3^3\cdot2^{10}$
and
$\Disc(E_f)=\Disc(f)/2^{20}$.
\end{lem}
\begin{proof}
We put $A=-I(f)/48$ and $B=-J(f)/1728$,
which are integers by \eqref{eq:cong-for-S}.
Then $E_f$ is given by $y^2=x^3+Ax+B$.
Suppose $p^4\mid A$ and $p^6\mid B$ for a prime $p$.
Since $2^5\nmid B$ by \eqref{eq:cong-for-S},
$p$ must be odd. But then
since $-4A^3-27B^2=\Disc(f)/2^{12}=2^8\Disc(f)/2^{20}$,
$\Disc(f)/2^{20}$ must be divisible by $p^{12}$ which is a contradiction.
Thus $(I(f),J(f))\in\mathcal E_\Z$ and this implies the first assertion.

We further let $a=A/2^4$
and \yicmtwo{$b=B/2^6-1/4$,}
which are still integers.
The elliptic curve $E_f$
is isomorphic to $y^2=x^3+Ax/2^4+B/2^6=x^3+ax+b+1/4$.
Replacing $y$ with $y+1/2$, this in turn is isomorphic to
$y^2+y=x^3+ax+b$. The discriminant for this
Weierstrass form is $\Delta=-64a^3-432b^2-216b-27$,
which is $\Disc(f)/2^{20}$.
Since $\Disc(E_f)=\Delta/m^{12}$ for an integer $m$,
by assumption $m=1$ and thus $\Disc(E_f)=\Disc(f)/2^{20}$.
\end{proof}

Finally, it remains to shift from the ${\rm GL}_2(\Z)$-orbits counted in Theorem \ref{thm:used_sieve} to
${\rm PGL}_2(\Q)$-orbits, which by \cite{BSD, BS2} (stated as Theorem \ref{thm:bsd-bs} here) count $2$-Selmer groups of elliptic curves.

Exactly the same issue occurs in \cite{BS2}. In \cite[Section 3.2]{BS2}, Bhargava and Shankar define $n(f)$ to be the number of $\GL_2(\Z)$-orbits
inside the $\PGL_2(\Q)$-orbit of any $f \in V(\Z)$. They also define a related quantity $m(f)$, satisfying $n(f) \leq m(f)$ for all $f$.
(They prove equality for almost all $f$, but we won't need this.)
As is proved in \cite[Proposition 3.6]{BS2},
$m(f)$ admits a decomposition $m(f) = \prod_p m_p(f)$, where
\begin{align*}
{\rm PGL}_2(\Q_p)_f := & \  \{g\in {\rm PGL}_2(\Q_p)\mid g\cdot f\in V(\Z_p)\},
\\
m_p(f) := & \ |{\rm PGL}_2(\Z_p)\backslash {\rm PGL}_2(\Q_p)_f|.
\end{align*}

Since $\Disc(f)$ is squarefree away from $2$,
\yicmtwo{by} \cite[Proposition 3.18]{BS2}, we have $m_p(f) = 1$ for all
odd primes. It thus remains to bound $m_2(f)$.
This is contained in the following general proposition,
which is actually a generalization of \cite[Proposition 3.18]{BS2}.
We note that the idea of the proof is also used
in Shankar and the second named author's work \cite{ST}
to prove a related property of the function $m_p$.

\begin{prop}
Let $\N=\{n\in\Z\mid n\geq0\}$ and $p$ be any prime.
There exists a function $M_p\colon \N\rightarrow\N$ 
with $M_p(0)=M_p(1)=1$ 
such that
$m_p(f)\leq M_p(v_p(\Disc(f)))$ for all $f\in V(\Z_p)$.
\end{prop}
\begin{proof}
For $k\in\N$, let
$\mathcal G_k={\rm PGL}_2(\Z_p)
{\footnotesize	\begin{pmatrix}p^k&0\\0&1\end{pmatrix}}
	{\rm PGL}_2(\Z_p)\subset {\rm PGL}_2(\Q_p)$,
where we use the same notation
${\footnotesize	\begin{pmatrix}p^k&0\\0&1\end{pmatrix}}$
for an element in $\gl_2(\Q_p)$ and its image in ${\rm PGL}_2(\Q_p)$.
Then it is well known that
${\rm PGL}_2(\Q_p)=\bigsqcup_{k\in\N}\mathcal G_k$ and that
$|{\rm PGL}_2(\Z_p)\backslash \mathcal G_k|$ is finite for each $k\in\N$.

Let $g\in\mathcal G_k$, $f\in V(\Z_p)$ and suppose $g\cdot f\in V(\Z_p)$.
Write $g=\gamma_1{\footnotesize	\begin{pmatrix}p^k&0\\0&1\end{pmatrix}}\gamma_2$
where $\gamma_1,\gamma_2\in {\rm PGL}_2(\Z_p)$ and $f'=\gamma_2 f\in V(\Z_p)$.
Then ${\footnotesize	\begin{pmatrix}p^k&0\\0&1\end{pmatrix}}f'\in V(\Z_p)$.
Let $f'=(a_0,a_1,a_2,a_3,a_4)$. Then
${\footnotesize	\begin{pmatrix}p^k&0\\0&1\end{pmatrix}}f'=
(a_0p^{2k},a_1p^k,a_2,a_3/p^k,a_4/p^{2k})$
and so $p^{k}\mid a_3$ and $p^{2k}\mid a_4$.
Since the polynomial $\Disc\in\Z[a_0,a_1,a_2,a_3,a_4]$
is in the ideal generated by $a_3^2$ and $a_4$, we have $p^{2k}\mid \Disc(f')$.
Therefore  $p^{2k}\mid \Disc(f)$ as well since $\Disc(f)=\Disc(f')$.

Therefore ${\rm PGL}_2(\Q_p)_f\subset \bigsqcup_{p^{2k}\mid \Disc(f)}\mathcal G_k$, and 
we have
\[
m_p(f)=|{\rm PGL}_2(\Z_p)\backslash {\rm PGL}_2(\Q_p)_f|
\leq\sum_{2k\leq {\rm ord}_p(\Disc(f))}|{\rm PGL}_2(\Z_p)\backslash \mathcal G_k|.
\]
Thus $M_p(l):=\sum_{k\leq l/2}|{\rm PGL}_2(\Z_p)\backslash \mathcal G_k|$
will do.
\end{proof}

We can now deduce Theorem \ref{thm:ap-2selmer}
from Theorem \ref{thm:used_sieve}.
Any $f\in S$ whose $\GL_2(\Z)$-equivalence
class is counted in Theorem \ref{thm:used_sieve}
is locally soluble and irreducible over $\Q$,
thus it corresponds to a non-trivial element in ${\rm Sel}_2(E_f)$.
By Lemma \ref{lem:E_f}, $\Disc(E_f)=\Disc(f)$,
and since $H(f)$ and $H(E_f)$ coincide up to a fixed scalar,
we may replace the condition $H(f) < X$ with $H(E_f) < X$.
Finally, $\Disc(f)/2^{20}$ is an odd squarefree integer,
we have $n(f)\leq m(f)=m_2(f)\leq M_2(20)$, asserting that
the number of $\GL_2(\Z)$-orbits in the ${\rm PGL}_2(\Q)$-equivalence
class of $f$ is absolutely bounded.
We thus obtain Theorem \ref{thm:ap-2selmer}.

\section*{Acknowledgments}
We thank Wei Ho, Zev Klagsbrun and Hiroyuki Ochiai
for helpful comments and suggestions.
We extend our gratitude to Fumihiro Sato for pointing out
an error in our citation of references.
We also thank Tetsushi Ito, whose suggestion provided the starting point
for our proof of Theorem \ref{thm:ff-disc}.

FT was partially supported by the National Science Foundation under Grant No. DMS-2101874, and by grants from the Simons Foundation (Nos. 563234 and 586594).
YI was supported by JSPS KAKENHI Grant Number 20K03747, 21K13773 and 21K18557.
TT was supported by JSPS KAKENHI Grant Number 17H02835, 21K18577 and 22H01115.
We would like to thank all of these agencies for their support.
This work is also supported by the Research Institute for Mathematical Sciences,
an International Joint Usage/Research Center located in Kyoto University.


  \footnotesize

   Ishitsuka: \textsc{Institute of Mathematics for Industry, Kyushu University, Fukuoka, 819-0395, Japan}\par\nopagebreak
   \textit{E-mail address}: \texttt{yishi1093@gmail.com}

  \medskip
    
  Taniguchi: \textsc{Department of Mathematics, Faculty of Science, Kobe University, Kobe, 657-8501, Japan}\par\nopagebreak
  \textit{E-mail address}: \texttt{tani@math.kobe-u.ac.jp}

  \medskip

  Thorne: \textsc{Department of Mathematics, University of South Carolina,
    Columbia, SC 29208}\par\nopagebreak
  \textit{E-mail address}: \texttt{thorne@math.sc.edu}

  \medskip

   Xiao: \textsc{Department of Mathematics and Statistics, University of Northern British Columbia, Prince George, British Columbia, Canada, V2N 4Z9}\par\nopagebreak
   \textit{E-mail address}: \texttt{StanleyYao.Xiao@unbc.ca}

  \medskip

\end{document}